\documentclass[a4paper]{amsart}
\usepackage{amsmath, amssymb, amsthm, mathrsfs, amscd}
\usepackage{graphicx}

\theoremstyle{plain} 
 \newtheorem{thm}{Theorem}[section]
 \newtheorem{lem}[thm]{Lemma}
 
 \newtheorem{prop}[thm]{Proposition}
 \newtheorem{claim}[thm]{Claim}

\theoremstyle{definition}
  \newtheorem{defn}[thm]{Definition}

\theoremstyle{remark}
  \newtheorem{rem}[thm]{Remark}

\renewcommand{\mod}{{\rm Mod}}
\newcommand{\cal}{\mathcal}
\newcommand{\ci}[2]{\cite[#1]{#2}}

\newcommand{\calc}{\mathcal{C}}
\newcommand{\calt}{\mathcal{T}}

\newcommand{\cald}{\mathcal{D}}
\renewcommand{\pmod}{{\rm PMod}}
\newcommand{\lk}{{\rm Lk}}
\newcommand{\cale}{\mathcal{E}}
\newcommand{\cala}{\mathcal{A}}
\newcommand{\calf}{\mathcal{F}}
\newcommand{\cali}{\mathcal{I}}

\begin{document}

\title[The Torelli complex for the one-holed genus two surface]{Automorphisms of the Torelli complex for the one-holed genus two surface}
\author{Yoshikata Kida}
\author{Saeko Yamagata}
\address{Department of Mathematics, Kyoto University, 606-8502 Kyoto, Japan}
\email{kida@math.kyoto-u.ac.jp}
\address{Faculty of Education and Human Sciences, Yokohama National University, 240-8501 Yokohama, Japan}
\email{yamagata@ynu.ac.jp}
\date{April 7, 2013}
%start: May 30

\begin{abstract}
Let $S$ be a connected, compact and orientable surface of genus two having exactly one boundary component.
We study automorphisms of the Torelli complex for $S$, and describe any isomorphism between finite index subgroups of the Torelli group for $S$.
More generally, we study superinjective maps from the Torelli complex for $S$ into itself, and show that any finite index subgroup of the Torelli group for $S$ is co-Hopfian.
\end{abstract}

\maketitle

%%%%%%%%%%%%%%%%%%%%%%%%%%%%%%%%%%%%%%%%%%%%%

\section{Introduction}

Let $S=S_{g, p}$ be a connected, compact and orientable surface of genus $g$ with $p$ boundary components.
Let $\mod^*(S)$ be the {\it extended mapping class group} of $S$, i.e., the group of isotopy classes of homeomorphisms from $S$ onto itself, where isotopy may move points of the boundary of $S$.
When $p\leq 1$, the {\it Torelli group} of $S$, denoted by $\cali(S)$, is defined as the subgroup of $\mod^*(S)$ consisting of all elements acting on the homology group $H_1(S, \mathbb{Z})$ trivially.
As a consequence of \cite{bm}, \cite{bm-add}, \cite{farb-ivanov} and \cite{kida-tor}, if $g\geq 3$ and $p\leq 1$, then any isomorphism between finite index subgroups of $\cali(S)$ is the conjugation by an element of $\mod^*(S)$.
One purpose of this paper is to obtain the same conclusion when $g=2$ and $p=1$.
A key step in the proof of these results is to describe any automorphism of the Torelli complex $\calt(S)$ of $S$, which is a simplicial complex on which $\mod^*(S)$ naturally acts.
The Torelli complex (with a certain marking) of a closed surface was first introduced by Farb-Ivanov \cite{farb-ivanov} to attack the same problem on the Torelli group.
The computation of automorphisms of $\calt(S)$ in our case is more delicate than that in the other cases.
One difficulty stems from lowness of the dimension of $\calt(S_{2, 1})$.
In fact, $\calt(S_{2, 1})$ is of dimension 1, and if $S$ is a surface dealt with in the cited references, then $\calt(S)$ is of dimension at least 2.
Details are discussed in Remark \ref{rem-low}.

Let us introduce terminology and notation to define the Torelli complex.
A simple closed curve in $S$ is called {\it essential} in $S$ if it is neither homotopic to a single point of $S$ nor isotopic to a boundary component of $S$.
Let $V(S)$ denote the set of isotopy classes of essential simple closed curves in $S$.
For $\alpha, \beta \in V(S)$, we define $i(\alpha, \beta)$ to be the {\it geometric intersection number} of $\alpha$ and $\beta$, i.e., the minimal cardinality of $A\cap B$ among representatives $A$ and $B$ of $\alpha$ and $\beta$, respectively.
Let $\Sigma(S)$ denote the set of non-empty finite subsets $\sigma$ of $V(S)$ with $i(\alpha, \beta)=0$ for any $\alpha, \beta \in \sigma$. 
We extend $i$ to the symmetric function on the square of $V(S)\sqcup \Sigma(S)$ with $i(\alpha, \sigma)=\sum_{\beta \in \sigma}i(\alpha, \beta)$ and $i(\sigma, \tau)=\sum_{\beta \in \sigma, \gamma \in \tau}i(\beta, \gamma)$ for any $\alpha \in V(S)$ and $\sigma, \tau \in \Sigma(S)$.

An essential simple closed curve $a$ in $S$ is called {\it separating} in $S$ if $S\setminus a$ is not connected.
Otherwise, $a$ is called {\it non-separating} in $S$.
These properties depend only on the isotopy class of $a$.
Let $V_s(S)$ denote the subset of $V(S)$ consisting of all elements whose representatives are separating in $S$.
We mean by a {\it bounding pair (BP)} in $S$ a pair of essential simple closed curves in $S$, $\{ a, b\}$, such that
\begin{itemize}
\item $a$ and $b$ are disjoint and non-isotopic;
\item each of $a$ and $b$ is non-separating in $S$; and
\item the surface obtained by cutting $S$ along $a\cup b$ is not connected
\end{itemize}
(see Figure \ref{fig-bp}).
%====================================
\begin{figure}
\begin{center}
\includegraphics[width=3cm]{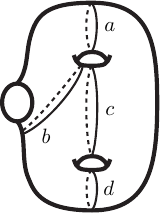}
\caption{The pair $\{ a, b \}$ is a BP. Any other pair of the four curves, $a$, $b$, $c$ and $d$, is not a BP.}\label{fig-bp}
\end{center}
\end{figure}
%====================================
These conditions depend only on the isotopy classes of $a$ and $b$.
Let $V_{bp}(S)$ denote the subset of $\Sigma(S)$ consisting of all elements that correspond to a BP in $S$.

\begin{defn}
The {\it Torelli complex} $\calt(S)$ of $S$ is defined as the abstract simplicial complex so that the set of vertices of $\calt(S)$ is the disjoint union $V_s(S) \sqcup V_{bp}(S)$, and a non-empty finite subset $\sigma$ of $V_s(S)\sqcup V_{bp}(S)$ is a simplex of $\calt(S)$ if and only if we have $i(\alpha, \beta)=0$ for any $\alpha, \beta \in \sigma$.
\end{defn}

For $\alpha \in V(S)$, let $t_{\alpha}\in \mod^*(S)$ denote the {\it (left) Dehn twist} about $\alpha$.
We note that if $p\leq 1$, then the Torelli group $\cali(S)$ contains $t_{\alpha}$ and $t_{\beta}t_{\gamma}^{-1}$ for any $\alpha \in V_s(S)$ and any $\{ \beta, \gamma \}\in V_{bp}(S)$, and is generated by all elements of these forms as discussed in \cite{johnson}.
This fact is a motivation for introducing the Torelli complex.

In this paper, we study not only automorphisms of $\calt(S_{2, 1})$ but also simplicial maps from $\calt(S_{2, 1})$ into itself satisfying strong injectivity, called superinjectivity.
We mean by a {\it superinjective map} from $\calt(S)$ into itself a simplicial map $\phi \colon \calt(S)\rightarrow \calt(S)$ satisfying $i(\phi(\alpha), \phi(\beta))\neq 0$ for any two vertices $\alpha$, $\beta$ of $\calt(S)$ with $i(\alpha, \beta)\neq 0$.
Any superinjective map from $\calt(S)$ into itself is shown to be injective (see \cite[Section 2.2]{kida-tor}).
Superinjectivity was first introduced by Irmak \cite{irmak1} for simplicial maps between the complexes of curves to study injective homomorphisms from a finite index subgroup of $\mod^*(S)$ into $\mod^*(S)$.
Our main result is the following:

\begin{thm}\label{main_thm}
We set $S=S_{2, 1}$.
Then the following assertions hold:
\begin{enumerate}
\item Any superinjective map from $\calt(S)$ into itself is induced by an element of $\mod^*(S)$. 
\item If $\Gamma$ is a finite index subgroup of $\cali(S)$ and if $f\colon \Gamma \rightarrow \cali(S)$ is an injective homomorphism, then there exists a unique element $\gamma_0$ of $\mod^*(S)$ with $f(\gamma)=\gamma_0\gamma \gamma_0^{-1}$ for any $\gamma \in \Gamma$. 
\end{enumerate}
\end{thm}

The process to derive assertion (ii) from assertion (i) is already discussed in \cite[Section 6.3]{kida-tor}.
We thus omit the proof of assertion (ii)

We say that a group $\Gamma$ is {\it co-Hopfian} if any injective homomorphism from $\Gamma$ into itself is surjective.
Assertion (ii) implies that any finite index subgroup of $\cali(S_{2, 1})$ is co-Hopfian.

\begin{rem}
Farb-Ivanov \cite{farb-ivanov} announced the computation of automorphisms of the Torelli geometry for a closed surface, which is the Torelli complex with a certain marking.
As its consequence, they also announce the result that if $S=S_{g, 0}$ is a surface with $g\geq 5$, then any isomorphism between finite index subgroups of $\cali(S)$ is induced by an element of $\mod^*(S)$. 
McCarthy-Vautaw \cite{mv} computed automorphisms of $\cali(S)$ for $S=S_{g, 0}$ with $g\geq 3$.
Brendle-Margalit \cite{bm}, \cite{bm-add} showed that any automorphism of $\calt(S)$ and any isomorphism between finite index subgroups of $\cali(S)$ are induced by an element of $\mod^*(S)$ when $S=S_{g, 0}$ with $g\geq 3$.
The same conclusion for $S=S_{g, p}$ with either $g=1$ and $p\geq 3$; $g=2$ and $p\geq 2$; or $g\geq 3$ and $p\geq 0$ as Theorem \ref{main_thm} was obtained in \cite{kida-cohop}, based on \cite{kida-tor}, where the Torelli group $\cali(S)$ is defined as the subgroup of $\mod^*(S)$ generated by all elements of the forms $t_{\alpha}$ with $\alpha \in V_s(S)$ and $t_{\beta}t_{\gamma}^{-1}$ with $\{ \beta, \gamma \}\in V_{bp}(S)$.

If $S=S_{2, 0}$, then $\calt(S)$ is zero-dimensional and consists of countably infinitely many vertices.
Moreover, $\cali(S)$ is known to be isomorphic to the non-abelian free group of infinite rank, due to Mess \cite{mess} (see \cite{bbm} for another proof).
It thus turns out that automorphisms of $\calt(S)$ and $\cali(S)$ are not necessarily induced by an element of $\mod^*(S)$.

If $S=S_{2, 1}$, then $\calt(S)$ is one-dimensional and connected.
The latter is proved by using the technique in \cite[Lemma 2.1]{putman-conn} to obtain connectivity of a simplicial complex on which $\mod^*(S)$ acts.
It also follows from Lemma \ref{seq_hexagon}.
\end{rem}

\begin{rem}\label{rem-low}
In \cite{kida-tor}, when either $g=1$ and $p\geq 3$; $g=2$ and $p\geq 2$; or $g\geq 3$ and $p\geq 0$, the first author observed simplices of $\calt(S)$ of maximal dimension and the links of simplices in $\calt(S)$ to prove that any superinjective map from $\calt(S)$ into itself preserves $V_s(S)$ and $V_{bp}(S)$, respectively.
On the other hand, when $g=2$ and $p=1$, this fact does not immediately follow from only observations on simplices and their links because $\calt(S)$ is one-dimensional.
This makes our case more delicate than the other cases.

We define $\calc_s(S)$ as the full subcomplex of $\calt(S)$ spanned by $V_s(S)$, and call it the {\it complex of separating curves} for $S$.
This complex brings another difference between our case and the other cases.
In \cite{bm}, \cite{bm-add} and \cite{kida-tor}, automorphisms of $\calt(S)$ are described by showing that any automorphism of $\calc_s(S)$ is induced by an element of $\mod^*(S)$.
On the other hand, $\calc_s(S_{2, 1})$ consists of countably infinitely many $\aleph_0$-regular trees, and thus has continuously many automorphisms.
This is a direct consequence of \cite[Theorem 7.1]{kls} (see also Theorem \ref{thm-tree}).
\end{rem}

The paper is organized as follows.
In Section \ref{sec-pre}, we collect terminology employed throughout the paper.
We recall the complex of curves for $S$, ideal triangulations of punctured surfaces considered by Mosher \cite{mosher} and basic results on them. 
Setting $S=S_{2, 1}$, through Sections \ref{sec-hex}--\ref{sec-type3}, we observe hexagons in $\calt(S)$, or equivalently, simple cycles in $\calt(S)$ of length 6.
In Section \ref{sec-aut}, applying results in those sections, we show that any superinjective map $\phi$ from $\calt(S)$ into itself preserves $V_s(S)$ and $V_{bp}(S)$, respectively, and is surjective.
We construct an automorphism $\Phi$ of the complex of curves for $S$ inducing $\phi$.
It is known that $\Phi$ is induced by an element of $\mod^*(S)$, due to Ivanov \cite{iva-aut} (see Theorem \ref{thm-iva}).
Theorem \ref{main_thm} (i) then follows.
In Appendix \ref{sec-app}, we prove that there exists no simple cycle in $\calt(S)$ of length at most 5.
Hexagons in $\calt(S)$ are thus simple cycles in $\calt(S)$ of minimal length.
This is a notable property of $\calt(S)$ although we do not use it to prove Theorem \ref{main_thm} (i).

%%%%%%%%%%%%%%%%%%%%%%%%%%%%%%%%%%%%%%%%%%%%

\section{Preliminaries}\label{sec-pre}

\subsection{Terminology}

Let $S$ be a connected, compact and orientable surface.
Unless otherwise stated, we assume that a surface satisfies these conditions.
Let us denote by $\mod(S)$ the {\it mapping class group} of $S$, i.e., the subgroup of $\mod^*(S)$ consisting of isotopy classes of orientation-preserving homeomorphisms from $S$ onto itself.
We define $\pmod(S)$ as the {\it pure mapping class group} of $S$, i.e., the subgroup of $\mod^*(S)$ consisting of isotopy classes of homeomorphisms from $S$ onto itself preserving an orientation of $S$ and preserving each boundary component of $S$ as a set.

We mean by a curve in $S$ either an essential simple closed curve in $S$ or its isotopy class if there is no confusion.
A surface homeomorphic to $S_{1, 1}$ is called a {\it handle}.
A surface homeomorphic to $S_{0, 3}$ is called a {\it pair of pants}.
Let $a$ be a separating curve in $S$.
If $a$ cuts off a handle from $S$, then $a$ is called an {\it h-curve} in $S$.
If $a$ cuts off a pair of pants from $S$, then $a$ is called a {\it p-curve} in $S$.
We call an element of $V_s(S)$ corresponding to an h-curve and a p-curve in $S$ an {\it h-vertex} and a {\it p-vertex}, respectively, and call an element of $V_{bp}(S)$ a {\it BP-vertex}.

Suppose that $\partial S$, the boundary of $S$, is non-empty.
Let $I$ be the closed unit interval.
We mean by an {\it essential simple arc} in $S$ the image of an injective continuous map $f\colon I\rightarrow S$ such that
\begin{itemize}
\item we have $f(\partial I)\subset \partial S$ and $f(I\setminus \partial I)\subset S\setminus \partial S$; and
\item there exists no closed disk $D$ embedded in $S$ and whose boundary is the union of $f(I)$ and an arc in $\partial S$.
\end{itemize}
The boundary of an essential simple arc $l$ is denoted by $\partial l$.
Let $V_a(S)$ denote the set of isotopy classes of essential simple arcs in $S$, where isotopy may move the end points of arcs, keeping them staying in $\partial S$.
We often identify an element of $V_a(S)$ with its representative if there is no confusion.

An essential simple arc $l$ in $S$ is called {\it separating} in $S$ if the surface obtained by cutting $S$ along $l$ is not connected.
Otherwise, $l$ is called {\it non-separating} in $S$.
These properties depend only on the isotopy class of $l$.

For $\sigma \in \Sigma(S)$, we mean by a {\it representative} of $\sigma$ the union of mutually disjoint representatives of elements in $\sigma$.
Given two elements $\alpha, \beta \in V(S)\sqcup \Sigma(S)$ and their representatives $A$, $B$, respectively, we say that $A$ and $B$ {\it intersect minimally} if we have $|A\cap B|=i(\alpha, \beta)$. 
For $\alpha, \beta \in V(S)\sqcup \Sigma(S)$, we say that $\alpha$ and $\beta$ are {\it disjoint} if $i(\alpha, \beta)=0$.
Otherwise, we say that $\alpha$ and $\beta$ {\it intersect}.
For an element $\alpha$ of $V(S)$ (or its representative), we denote by $S_{\alpha}$ the surface obtained by cutting $S$ along $\alpha$.
Similarly, for an element $\sigma$ of $\Sigma(S)$ (or its representative), we denote by $S_{\sigma}$ the surface obtained by cutting $S$ along all curves in $\sigma$.
Each component of $S_{\sigma}$ is often identified with a complementary component in $S$ of a tubular neighborhood of a one-dimensional submanifold representing $\sigma$ if there is no confusion.
For any component $Q$ of $S_{\sigma}$, the set $V(Q)$ is naturally identified with a subset of $V(S)$.

%%%%%%%%%%%%%%%%%%%%%%%%%%%%%%%%%%%%%%%%%%%%

\subsection{The complex of curves}

In the proof of Theorem \ref{main_thm} (i), we use a result on automorphisms of the complex of curves.
The {\it complex of curves} for a surface $S$, denoted by $\calc(S)$, is defined as the abstract simplicial complex so that the sets of vertices and simplices of $\calc(S)$ are $V(S)$ and $\Sigma(S)$, respectively.

\begin{thm}[\ci{Theorem 1}{iva-aut}]\label{thm-iva}
If $S=S_{g, p}$ is a surface with $g\geq 2$ and $p\geq 0$, then any automorphism of $\calc(S)$ is induced by an element of $\mod^*(S)$.
\end{thm}

We refer to \cite{kork-aut} and \cite{luo} for similar results for other surfaces.
Theorem \ref{main_thm} (i) is obtained by showing that when $S=S_{2, 1}$, for any superinjective map $\phi \colon \calt(S)\rightarrow \calt(S)$, there exists an automorphism $\Phi$ of $\calc(S)$ inducing $\phi$, that is, satisfying the equalities
\[\Phi(\alpha)=\phi(\alpha)\quad \textrm{and}\quad \{ \Phi(\beta), \Phi(\gamma)\} =\phi(\{ \beta, \gamma \})\]
for any $\alpha \in V_s(S)$ and any $\{ \beta, \gamma \}\in V_{bp}(S)$.

We note that the complex of separating curves for $S$, defined in Remark \ref{rem-low} and denoted by $\calc_s(S)$, is the full subcomplex of $\calc(S)$ spanned by $V_s(S)$.

%%%%%%%%%%%%%%%%%%%%%%%%%%%%%%%%%%%%%%%%%%%%

\subsection{Ideal triangulations of a punctured surface}

We recall basic properties of ideal triangulations of a punctured surface discussed by Mosher \cite{mosher}, which will be used only in the proof of Lemma \ref{surj_lkd}.
Let $S$ be a closed surface of positive genus $g$, and let $P$ be a non-empty finite subset of $S$.
The pair $(S, P)$ is then called a {\it punctured surface}.
Let $I$ be the closed unit interval.
We mean by an {\it ideal arc} in $(S, P)$ the image of a continuous map $f\colon I\rightarrow S$ such that
\begin{itemize}
\item we have $f(\partial I)\subset P$ and $f(I\setminus \partial I)\subset S\setminus P$;
\item $f$ is injective on $I\setminus \partial I$; and
\item there exists no closed disk $D$ embedded in $S$ with $\partial D=f(I)$ and $(D\setminus \partial D)\cap P=\emptyset$.
\end{itemize}
Two ideal arcs $l_1$, $l_2$ in $(S, P)$ are called {\it isotopic} if we have $l_1\cap P=l_2\cap P$; and $l_1$ and $l_2$ are isotopic relative to $l_1\cap P$ as arcs in $(S\setminus P)\cup (l_1\cap P)$.
We mean by an {\it ideal triangulation} of $(S, P)$ a cell division $\delta$ of $S$ such that
\begin{enumerate}
\item[(a)] the set of 0-cells of $\delta$ is $P$;
\item[(b)] any 1-cell of $\delta$ is an ideal arc in $(S, P)$; and
\item[(c)] any 2-cell of $\delta$ is a {\it triangle}, that is, it is obtained by attaching a Euclidean triangle $\tau$ to the 1-skeleton of $\delta$, mapping each vertex of $\tau$ to a 0-cell of $\delta$, and each edge of $\tau$ to a 1-cell of $\delta$.
\end{enumerate}
The following properties are noticed in \cite[p.14]{mosher}.

\begin{lem}\label{lem-mosher}
The following assertions hold:
\begin{enumerate}
\item Any cell division of $S$ satisfying conditions (a) and (c) in the definition of an ideal triangulation necessarily satisfies condition (b).
\item Let $\delta$ be an ideal triangulation of $(S, P)$.
Then any two distinct 1-cells of $\delta$ are not isotopic.
\end{enumerate}
\end{lem}

Let $R$ be a surface of genus $g$ with $|P|$ boundary components.
Suppose that $S$ is obtained from $R$ by shrinking each component of $\partial R$ into a point, and that $P$ is the set of points into which components of $\partial R$ are shrunken.
The natural map from $R$ onto $S$ induces the bijection from $V_a(R)$ onto the set of isotopy classes of ideal arcs in $(S, P)$.

%%%%%%%%%%%%%%%%%%%%%%%%%%%%%%%%%%%%%%%%%%%%

\section{Non-existence of some hexagons}\label{sec-hex}

Let $\cal{G}$ be a simplicial complex.
We mean by a {\it hexagon} in $\cal{G}$ the full subcomplex of $\cal{G}$ spanned by six vertices $v_1,\ldots, v_6$ such that for any $j$ mod 6, $v_j$ and $v_{j+1}$ are adjacent; $v_j$ and $v_{j+2}$ are not adjacent; and $v_j$ and $v_{j+3}$ are not adjacent.
In this case, we say that the hexagon is defined by the 6-tuple $(v_1,\ldots, v_6)$.

Throughout this section, we set $S=S_{2, 1}$.
Examples of hexagons in $\calt(S)$ are described in Sections \ref{sec-type1}--\ref{sec-type3}.
In this section, we show that there exists no hexagon in $\calt(S)$ containing at most one BP-vertex.
Note that any separating curve in $S$ is an h-curve in $S$, and that any edge of $\calt(S)$ consists of either two h-vertices or an h-vertex and a BP-vertex (see Figure \ref{fig-edge}).
%====================================
\begin{figure}
\begin{center}
\includegraphics[width=7cm]{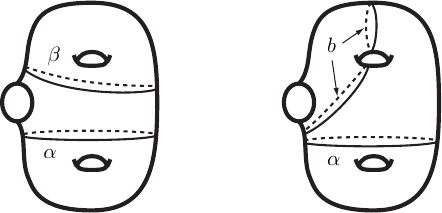}
\caption{Each of $\{ \alpha, \beta \}$ and $\{ \alpha, b\}$ is an edge of $\calt(S)$.}\label{fig-edge}
\end{center}
\end{figure}
%====================================
It follows that the number of BP-vertices of a hexagon in $\calt(S)$ is at most 3.

\begin{lem}\label{6-h_hex}
There exists no hexagon in $\calt(S)$ consisting of only h-vertices. 
\end{lem}

To prove this lemma, we use the following:

\begin{thm}[\ci{Theorem 7.1}{kls}]\label{thm-tree}
Let $S=S_{2, 1}$ be a surface, and let $\bar{S}$ denote the closed surface obtained from $S$ by attaching a disk to the boundary of $S$.
We define
\[\pi \colon \calc(S)\rightarrow \calc(\bar{S})\]
as the simplicial map associated to the inclusion of $S$ into $\bar{S}$.
Then for any vertex $\alpha$ of $\calc(\bar{S})$, the full subcomplex of $\calc(S)$ spanned by $\pi^{-1}(\alpha )$ is a tree.
\end{thm}

\begin{proof}[Proof of Lemma \ref{6-h_hex}]
We note that $\pi$ sends two adjacent h-vertices of $\calc(S)$ to the same vertex.
If there were a hexagon $\Pi$ in $\calt(S)$ consisting of only h-vertices, then $\pi$ would send $\Pi$ to a single vertex.
This contradicts Theorem \ref{thm-tree}.
\end{proof}

\begin{lem}\label{one-BP_hex}
There exists no hexagon in $\calt(S)$ containing exactly one BP-vertex. 
\end{lem}

\begin{proof}
Suppose that there exists such a hexagon $\Pi$ in $\calt(S)$. 
Let $(a, b, c, d, e, f)$ be a 6-tuple defining $\Pi$ with $a$ a BP-vertex. 
We then have the equality $\pi(b)=\pi(c)=\pi(d)=\pi(e)=\pi(f)$.
The curves $b$ and $f$ are in the component of $S_a$ that does not contain $\partial S$.
The equality $\pi(b)=\pi(f)$ thus implies the equality $b=f$.
This is a contradiction.
\end{proof}

%%%%%%%%%%%%%%%%%%%%%%%%%%%%%%%%%%%%%%%%%

\section{Hexagons of type 1}\label{sec-type1}

Throughout this section, we set $S=S_{2, 1}$.
We say that a hexagon in $\calt(S)$ is of {\it type 1} if it is defined by a 6-tuple $(v_1,\ldots, v_6)$ such that $v_1$, $v_3$, $v_4$ and $v_5$ are h-vertices and $v_2$ and $v_6$ are BP-vertices.
To construct such a hexagon in $\calt(S)$, we recall a hexagon in $\calc_s(S_{1, 3})$ (see Figure \ref{fig_hex_first}).
%====================================
\begin{figure}
\begin{center}
\includegraphics[width=12cm]{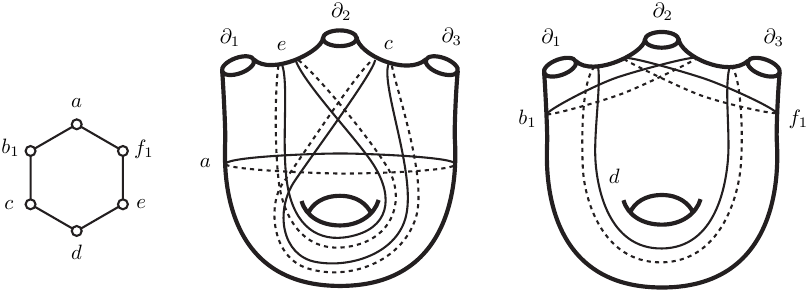}
\caption{The 6-tuple $(a, b_1, c, d, e, f_1)$ of the above curves defines a hexagon in $\calc_s(S_{1, 3})$. Let $S=S_{2, 1}$ be a surface, and let $\alpha$ be a non-separating curve in $S$. If $S_{\alpha}$ is drawn as above so that $\partial_1$ and $\partial_3$ correspond to $\alpha$ and $\partial_2$ corresponds to $\partial S$, then the 6-tuple $(a, b, c, d, e, f)$ with $b=\{ \alpha, b_1\}$ and $f=\{ \alpha, f_1\}$ defines a hexagon in $\calt(S)$ of type 1.}\label{fig_hex_first}
\end{center}
\end{figure}
%====================================
A fundamental property of hexagons in $\calc_s(S_{1, 3})$ is the following:

\begin{thm}[\ci{Theorem 5.2}{kida-tor}]\label{thm-13}
We set $X=S_{1, 3}$.
Then any two hexagons in $\calc_s(X)$ are sent to each other by an element of $\pmod(X)$.
\end{thm}

We now present a hexagon in $\calt(S)$ of type 1.
Let $\alpha$ be a non-separating curve in $S$.
Note that $S_\alpha$ is homeomorphic to $S_{1, 3}$. 
We define a simplicial map
\[\lambda_\alpha \colon \calc_s(S_\alpha) \rightarrow \calt(S)\]
as follows. 
Pick $\beta \in V_s(S_{\alpha})$. 
If the two boundary components of $S_\alpha$ corresponding to $\alpha$ are contained in distinct components of $S_{\{\alpha, \beta \}}$, then we have $\{\alpha, \beta\}\in V_{bp}(S)$ and set $\lambda_{\alpha}(\beta)=\{ \alpha, \beta \}$.
Otherwise, we have $\beta \in V_s(S)$ and set $\lambda_{\alpha}(\beta)=\beta$.
The map $\lambda_{\alpha}$ is superinjective, that is, for any $\gamma, \delta \in V_s(S_{\alpha})$, we have $i(\lambda_{\alpha}(\gamma), \lambda_{\alpha}(\delta))=0$ if and only if $i(\gamma, \delta)=0$.
Sending a hexagon in $\calc_s(S_{\alpha})$ through $\lambda_{\alpha}$, we obtain a hexagon in $\calt(S)$ of type 1 as precisely described in Figure \ref{fig_hex_first}.

The following theorem says that any hexagon in $\calt(S)$ of type 1 can be obtained through the above procedure.

\begin{thm}\label{thm-type1}
The following assertions hold:
\begin{enumerate}
\item For any hexagon $\Pi$ in $\calt(S)$ of type 1, there exist a non-separating curve $\alpha$ in $S$ and a hexagon $\Pi_0$ in $\calc_s(S_{\alpha})$ with $\lambda_{\alpha}(\Pi_0)=\Pi$.
\item Any two hexagons in $\calt(S)$ of type 1 are sent to each other by an element of $\mod(S)$.
\end{enumerate}
\end{thm}

\begin{proof}
Assertion (ii) follows from assertion (i) and Theorem \ref{thm-13}.
To prove assertion (i), we pick a hexagon $\Pi$ in $\calt(S)$ of type 1.
Let $(a, b, c, d, e, f)$ be a 6-tuple defining $\Pi$ with $b$ and $f$ BP-vertices. 
We choose representatives $A,\ldots, F$ of $a,\ldots, f$, respectively, such that any two of them intersect minimally.

Let $R$ denote the component of $S_C$ that is not a handle. 
Since $B$ is a BP in $R$ and is disjoint from $A$, the intersection $A\cap R$ consists of mutually isotopic, essential simple arcs in $R$ which are non-separating in $R$ (see Figure \ref{fig-arc-pf} (a)).
%====================================
\begin{figure}
\begin{center}
\includegraphics[width=12cm]{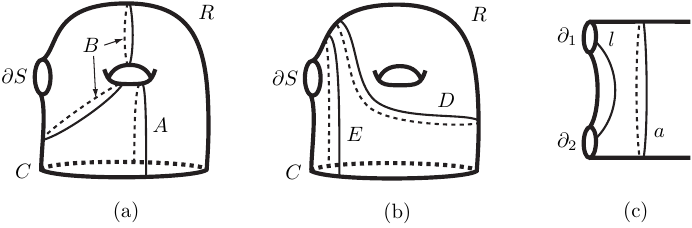}
\caption{}\label{fig-arc-pf}
\end{center}
\end{figure}
%====================================
Since $D$ is an h-curve in $R$ and is disjoint from $E$, the intersection $E\cap R$ consists of mutually isotopic, essential simple arcs in $R$ which are separating in $R$ (see Figure \ref{fig-arc-pf} (b)).
Let $l_1$ be a component of $A\cap R$, and let $l_2$ be a component of $E\cap R$.
If $l_1$ and $l_2$ could not be disjoint after isotopy which may move the end points of arcs, keeping them staying in $\partial R$, then the union of a subarc of $l_1$ and a subarc of $l_2$ would be a simple closed curve isotopic to $\partial S$. 
This is a contradiction because no simple closed curve in the component of $S_F$ that is not a pair of pants is isotopic to $\partial S$ as a curve in $S$. 
It thus turns out that $l_1$ and $l_2$ can be disjoint after isotopy. 
Note that $B$ consists of two boundary components of a regular neighborhood of $l_1\cup C$ in $R$, and that $D$ is a boundary component of a regular neighborhood of $l_2\cup C$ in $R$. 
It follows that exactly one component of $B$ is contained in the handle $Q$ cut off by $D$ from $S$.
We denote by $\alpha$ the isotopy class of that component of $B$.

Similarly, considering the component of $S_E$ that is not a handle, instead of that of $S_C$, we can show that exactly one component of $F$ is contained in $Q$. 
Let $\beta$ denote the isotopy class of that component of $F$. 
Since $B$ and $F$ are disjoint from $A\cap Q$, that consists of essential simple arcs in the handle $Q$, we have $\alpha=\beta$. 
We define two curves $b_1$, $f_1$ so that $b=\{ \alpha, b_1\}$ and $f=\{ \alpha, f_1\}$.
Any of $a$, $c$ and $e$ is disjoint from $\alpha$ because any of them is disjoint from $b$ or $f$.
The h-curve $d$ is also disjoint from $\alpha$ because $\alpha$ is the isotopy class of a curve in $Q$. 
The map $\lambda_\alpha$ sends the hexagon in $\calc_s(S_\alpha)$ defined by the 6-tuple $(a, b_1, c, d, e, f_1)$ to $\Pi$.
Assertion (i) is proved.
\end{proof}

Let $\cal{G}$ be a simplicial graph and $n$ a positive integer.
We mean by an {\it $n$-path} in $\cal{G}$ a subgraph of $\cal{G}$ obtained as the image of a simple path in $\cal{G}$ of length $n$ starting and terminating at vertices of $\cal{G}$.
In the rest of this section, we observe two hexagons in $\calt(S)$ of type 1 sharing a 3-path.

\begin{lem}\label{lem-hex1-hex1}
If $\Pi$ and $\Omega$ are hexagons in $\calt(S)$ of type 1 such that $\Pi \cap \Omega$ contains a 3-path, then we have $\Pi =\Omega$. 
\end{lem}

To prove this lemma, we make the following observation on hexagons in $\calc_s(S_{1, 3})$.

\begin{lem}\label{lem-13}
We set $X=S_{1, 3}$.
Let $H$ be a hexagon in $\calc_s(X)$.
Then for any 3-path $L$ in $H$, $H$ is the only hexagon in $\calc_s(X)$ containing $L$.
\end{lem}

Before proving this lemma, we introduce terminology.
Let $Y=S_{g, p}$ be a surface with $p\geq 2$.
For an essential simple arc $l$ in $Y$ and two distinct components $\partial_1$, $\partial_2$ of $\partial Y$, we say that $l$ {\it connects $\partial_1$ and $\partial_2$} (or {\it connects $\partial_1$ with $\partial_2$}) if one of the end points of $l$ lies in $\partial_1$ and another in $\partial_2$.

Suppose either $g\geq 1$ and $p\geq 2$ or $g=0$ and $p\geq 5$.
There is a one-to-one correspondence between elements of $V_s(Y)$ whose representatives are p-curves in $Y$ and elements of $V_a(Y)$ whose representatives connect two distinct components of $\partial Y$.
More specifically, for any p-curve $a$ in $Y$, we have an essential simple arc in $Y$ contained in the pair of pants cut off by $a$ from $Y$ and connecting two distinct components of $\partial Y$, which uniquely exists up to isotopy.
Conversely, for any essential simple arc $l$ in $Y$ connecting two distinct components $\partial_1$, $\partial_2$ of $\partial Y$, we have the p-curve in $Y$ that is a boundary component of a regular neighborhood of $l\cup \partial_1\cup \partial_2$ in $Y$ (see Figure \ref{fig-arc-pf} (c)).

\begin{proof}[Proof of Lemma \ref{lem-13}]
Let $(a, b, c, d, e, f)$ be a 6-tuple defining $H$ such that $a$, $c$ and $e$ are h-curves in $X$ and $b$, $d$ and $f$ are p-curves in $X$.
To prove the lemma, it is enough to show that $H$ is the only hexagon in $\calc_s(X)$ containing $a$, $b$, $c$ and $d$.

Choose representatives $A,\ldots, F$ of $a,\ldots, f$, respectively, such that any two of them intersect minimally.
We can then find essential simple arcs $l_B$, $l_D$ and $l_F$ in $X$ such that
\begin{itemize}
\item for any $G\in \{ B, D, F\}$, the arc $l_G$ lies in the pair of pants cut off by $G$ from $X$, and connects two distinct components of $\partial X$;
\item the arcs $l_B$, $l_D$ and $l_F$ are pairwise disjoint; and
\item any of $A\cap l_D$, $C\cap l_F$ and $E\cap l_B$ consists of exactly two points
\end{itemize}
(see Figure \ref{fig-arc13-pf} (a)).
%====================================
\begin{figure}
\begin{center}
\includegraphics[width=9cm]{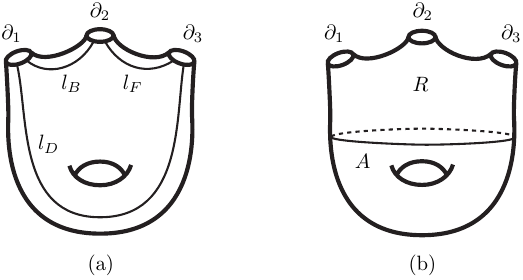}
\caption{}\label{fig-arc13-pf}
\end{center}
\end{figure}
%====================================
Label components of $\partial X$ as $\partial_1$, $\partial_2$ and $\partial_3$ so that $l_B$ connects $\partial_1$ and $\partial_2$ and $l_D$ connects $\partial_1$ and $\partial_3$.
Let $R$ denote the component of $X_A$ homeomorphic to $S_{0, 4}$, and let $\partial_4$ denote the component of $\partial R$ corresponding to $A$ (see Figure \ref{fig-arc13-pf} (b)). 
The intersection $l_D\cap R$ then consists of an arc connecting $\partial_1$ with $\partial_4$ and an arc connecting $\partial_3$ with $\partial_4$.
If we cut $R$ along $l_B$ and $l_D\cap R$, then we obtain a disk $K$ such that each of $\partial_2$ and $\partial_3$ corresponds to a single arc in $\partial K$.
It follows that up to isotopy, there exists at most one simple arc in $X$ connecting $\partial_2$ with $\partial_3$, meeting $\partial X$ only at its end points, and disjoint from $A$, $l_B$ and $l_D$.

We proved that any hexagon in $\calc_s(X)$ containing $a$, $b$, $c$ and $d$ contains $f$.
The lemma follows because $e$ is the only separating curve in $X$ disjoint from $d$ and $f$.
\end{proof}

\begin{proof}[Proof of Lemma \ref{lem-hex1-hex1}]
Let $\Pi$ and $\Omega$ be hexagons in $\calt(S)$ of type 1 such that $\Pi \cap \Omega$ contains a 3-path.
Let $(a, b, c, d, e, f)$ be a 6-tuple defining $\Pi$ with $b$ and $f$ BP-vertices.
We define $\alpha$ as the curve contained in $b$ and $f$.
The number of BP-vertices in $\Pi \cap \Omega$ is either 1 or 2.
If $\Pi \cap \Omega$ has two BP-vertices, then both $\Pi$ and $\Omega$ are hexagons in $\lambda_{\alpha}(\calc_s(S_{\alpha}))$, where $\lambda_{\alpha}\colon \calc_s(S_{\alpha})\rightarrow \calt(S)$ is the simplicial map defined right after Theorem \ref{thm-13}.
The equality $\Pi =\Omega$ holds by Lemma \ref{lem-13}.

Assuming that $\Pi \cap \Omega$ contains only one BP-vertex, we deduce a contradiction.
Without loss of generality, we may assume that $b$ is contained in $\Pi \cap \Omega$.
It then follows that $c$ and $d$ are also contained in $\Pi \cap \Omega$.
Since $\alpha$ is determined as the curve in $b$ disjoint from $d$, the two BPs in $\Omega$ share $\alpha$.
Both $\Pi$ and $\Omega$ are hexagons in $\lambda_{\alpha}(\calc_s(S_{\alpha}))$, and the equality $\Pi =\Omega$ holds by Lemma \ref{lem-13}.
This contradicts our assumption.
\end{proof}

%%%%%%%%%%%%%%%%%%%%%%%%%%%%%%%%%%%%%%%%

\section{Hexagons of type 2}\label{sec-type2}

Throughout this section, we set $S=S_{2, 1}$.
We say that a hexagon in $\calt(S)$ is of {\it type 2} if it is defined by a 6-tuple $(v_1,\ldots, v_6)$ such that $v_2$, $v_3$, $v_5$ and $v_6$ are h-vertices and $v_1$ and $v_4$ are BP-vertices.
We construct a hexagon of type 2 by gluing two pentagons in the Torelli complex of $S_{1, 3}$.

Let $\cal{G}$ be a simplicial complex.
We mean by a {\it pentagon} in $\cal{G}$ the full subcomplex of $\cal{G}$ spanned by five vertices $v_1,\ldots, v_5$ such that for any $j$ mod 5, $v_j$ and $v_{j+1}$ are adjacent, and $v_j$ and $v_{j+2}$ are not adjacent.
In this case, we say that the pentagon is defined by the 5-tuple $(v_1,\ldots, v_5)$.

Fix a non-separating curve $\delta$ in $S$, and let $X$ be the surface obtained by cutting $S$ along $\delta$, which is homeomorphic to $S_{1, 3}$.
Let $\partial_{\delta}^1$ and $\partial_{\delta}^2$ denote the two boundary components of $X$ corresponding to $\delta$.
%====================================
\begin{figure}
\begin{center}
\includegraphics[width=12cm]{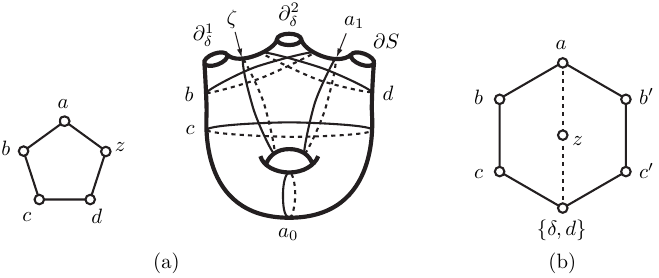}
\caption{}\label{fig-pen}
\end{center}
\end{figure}
%====================================
We have the pentagon in $\calt(X)$ defined by the 5-tuple $(a, b, c, d, z)$ in Figure \ref{fig-pen} (a), where we put $a=\{ a_0, a_1\}$ and $z=\{ a_0, \zeta \}$.

Fix a non-zero integer $m$, and put $b'=t_{\zeta}^m(b)$ and $c'=t_{\zeta}^m(c)$.
We then have the hexagon $\Pi$ in $\calt(S)$ defined by the 6-tuple $(a, b, c, \{ \delta, d\}, c', b')$, where $a$ and $\{ \delta, d\}$ are BPs in $S$, and $b$, $c$, $c'$ and $b'$ are h-curves in $S$ (see Figure \ref{fig-pen} (b)).
Note that $z$ is not a vertex of $\calt(S)$.
Let $n$ be a non-zero integer distinct from $m$, and put $b''=t_{\zeta}^n(b)$ and $c''=t_{\zeta}^n(c)$.
The hexagon in $\calt(S)$ defined by the 6-tuple $(a, b, c, \{ \delta, d\}, c'', b'')$ is distinct from $\Pi$ and shares a 3-path with $\Pi$.
This property is in contrast with Lemma \ref{lem-hex1-hex1} on hexagons of type 1.

The aim of this section is to show that any hexagon in $\calt(S)$ of type 2 can be obtained through this construction, and to describe the number of hexagons sharing a 3-path with a given hexagon of type 1 or type 2.
Uniqueness of the pentagon in $\calt(S_{1, 3})$ in Figure \ref{fig-pen} (a) is proved in the following:

\begin{lem}\label{lem-hex2-pen}
We set $X=S_{1, 3}$.
Then the following assertions hold:
\begin{enumerate}
\item Any pentagon in $\calt(X)$ having exactly two BP-vertices is defined by a 5-tuple $(v_1,\ldots, v_5)$ with $v_1$ and $v_5$ BP-vertices, $v_2$ and $v_4$ p-vertices, and $v_3$ an h-vertex.
\item Any two pentagons in $\calt(X)$ having exactly two BP-vertices are sent to each other by an element of $\mod(X)$.
\end{enumerate}
\end{lem}

To prove this lemma, we need uniqueness of pentagons in the one-dimensional complex $\calc(S_{0, 5})$.

\begin{lem}\label{lem-05}
We set $T=S_{0, 5}$.
Then for any two 5-tuples $(u_1,\ldots, u_5)$, $(v_1,\ldots, v_5)$ defining pentagons in $\calc(T)$, there exists an element $g$ of $\mod(T)$ with $g(u_j)=v_j$ for any $j=1,\ldots, 5$.
\end{lem}

\begin{proof}
As noticed right before the proof of Lemma \ref{lem-13}, there is a one-to-one correspondence between isotopy classes of curves in $T$ and isotopy classes of essential simple arcs in $T$ connecting two distinct components of $\partial T$.
Let $(u_1,\ldots, u_5)$ be a 5-tuple defining a pentagon in $\calc(T)$.
For each $j=1,\ldots, 5$, let $l_j$ be an essential simple arc in $T$ corresponding to $u_j$.

We claim that for any $j$ mod $5$, $l_j$ and $l_{j+2}$ can be isotoped so that they are disjoint, and exactly one component of $\partial T$, denoted by $\partial_j$, contains a point of $\partial l_j$ and a point of $\partial l_{j+2}$ (see Figure \ref{fig-arc5} (a)).
%====================================
\begin{figure}
\begin{center}
\includegraphics[width=12cm]{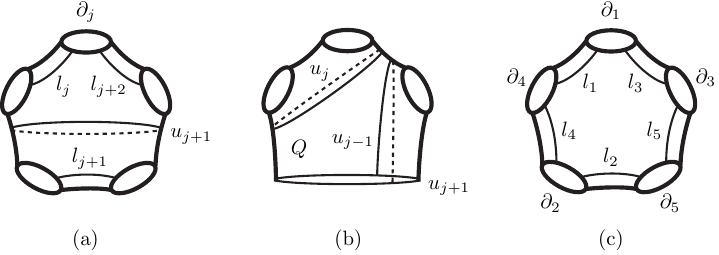}
\caption{}\label{fig-arc5}
\end{center}
\end{figure}
%====================================
Although this follows from \cite[Theorem 3.2]{kork-aut} or \cite[Lemma 4.2]{luo}, we give a proof for the reader's convenience.
Fix $j=1,\ldots, 5$.
The indices are regarded as numbers modulo $5$.
Let $Q$  be the component of $T_{u_{j+1}}$ homeomorphic to $S_{0, 4}$.
The curves $u_j$ and $u_{j+2}$ lie in $Q$.
Since $u_{j-1}$ is disjoint from $u_j$, the intersection $u_{j-1}\cap Q$ consists of mutually isotopic, essential simple arcs in $Q$ (see Figure \ref{fig-arc5} (b)).
Since $u_{j+3}$ is disjoint from $u_{j+2}$, the intersection $u_{j+3}\cap Q$ also consists of mutually isotopic, essential simple arcs in $Q$.
Any component of $u_{j-1}\cap Q$ and any component of $u_{j+3}\cap Q$ are not isotopic because otherwise we would have $u_j=u_{j+2}$.
Our claim then follows because $u_{j-1}$ and $u_{j+3}$ are disjoint.

We may therefore assume that $l_1,\ldots, l_5$ are mutually disjoint.
We next claim that $\partial_1,\ldots, \partial_5$ are mutually distinct.
For any $j$ mod 5, $\partial_j$ and $\partial_{j+1}$ are distinct because they are contained in the pairs of pants cut off by the curves $u_j$ and $u_{j+1}$, respectively, that are disjoint and distinct.
For any $j$ mod 5, $\partial_j$ and $\partial_{j+2}$ are distinct because they are contained in the pairs of pants cut off by the curves $u_j$ and $u_{j+4}$, respectively, that are disjoint and distinct.
The claim follows.

Let $(v_1,\ldots, v_5)$ be a 5-tuple defining a pentagon in $\calc(T)$.
For each $j=1,\ldots, 5$, we choose an essential simple arc $r_j$ in $T$ corresponding to $v_j$ so that $r_1,\ldots, r_5$ are mutually disjoint.
Applying an element of $\mod(T)$ to $(v_1,\ldots, v_5)$, we may assume that for any $j$ mod 5, $\partial_j$ contains a point of $\partial r_j$ and a point of $\partial r_{j+2}$.
Cutting $T$ along $\bigcup_{j=1}^5l_j$, we obtain two disks.
The boundary of each of those disks consists of arcs contained in
\[\partial_1,\ l_1,\ \partial_4,\ l_4,\ \partial_2,\ l_2,\ \partial_5,\ l_5,\ \partial_3,\ l_3,\]
along the boundary (see Figure \ref{fig-arc5} (c)).
The same property holds for the arcs $r_1,\ldots, r_5$.
We can thus find a homeomorphism of $T$ onto itself sending $\partial_j$ to itself and sending $l_j$ to $r_j$ for any $j=1,\ldots, 5$.
The lemma is proved.
\end{proof}

\begin{proof}[Proof of Lemma \ref{lem-hex2-pen}]
To prove assertion (i), we use the following properties:
\begin{enumerate}
\item[(1)] The link of any BP-vertex in $\calt(X)$ consists of BP-vertices and p-vertices.
\item[(2)] The link of any p-vertex in $\calt(X)$ consists of BP-vertices and h-vertices.
\end{enumerate}
Let $P$ be a pentagon in $\calt(X)$ with exactly two BP-vertices.
If the two BP-vertices of $P$ were not adjacent, then property (1) would imply that the other three vertices of $P$ are p-vertices.
We then have two adjacent p-vertices of $P$, and this contradicts property (2).
It follows that the two BP-vertices of $P$ are adjacent.
Properties (1) and (2) imply assertion (i).

To prove assertion (ii), we pick two pentagons $P$, $P'$ in $\calt(X)$ having exactly two BP-vertices.
Let $(a, b, c, d, e)$ be a 5-tuple defining $P$ with $a$ and $e$ BP-vertices.
Let $(a', b', c', d', e')$ be a 5-tuple defining $P'$ with $a'$ and $e'$ BP-vertices.
Any two distinct and disjoint BPs in $X$ have a common curve in $X$.
Let $\alpha$ be the curve in $a\cap e$.
We define curves $a_1$ and $e_1$ in $X$ so that $a=\{ \alpha, a_1\}$ and $e=\{ \alpha, e_1\}$.
We may assume that $\alpha$ is also the curve in $a'\cap e'$ after applying an element of $\pmod(X)$ to $P'$.
We define curves $a_1'$ and $e_1'$ in $X$ so that $a'=\{ \alpha, a_1'\}$ and $e'=\{ \alpha, e_1'\}$.
The two p-curves $b$ and $d$ fill the component of $S_c$ homeomorphic to $S_{0, 4}$.
Since $\alpha$ is disjoint from $b$ and $d$, the curve $\alpha$ is disjoint from $c$.
Similarly, $\alpha$ is disjoint from $b'$, $c'$ and $d'$.

Let $X_{\alpha}$ be the surface obtained by cutting $X$ along $\alpha$, which is homeomorphic to $S_{0, 5}$.
Each of the 5-tuples $(a_1, b, c, d, e_1)$ and $(a_1', b', c', d', e_1')$ defines a pentagon in $\calc(X_{\alpha})$.
By Lemma \ref{lem-05}, we obtain an element $g$ of $\mod(X_{\alpha})$ sending $(a_1, b, c, d, e_1)$ to $(a_1', b', c', d', e_1')$.
The two boundary components of $X_{\alpha}$ corresponding to $\alpha$ lie in the pair of pants cut off by $c$ from $X_{\alpha}$ and in that cut off by $c'$ from $X_{\alpha}$.
The equality $g(c)=c'$ implies that $g$ preserves those two boundary components of $X_{\alpha}$.
Assertion (ii) follows.
\end{proof}

We now present several properties of hexagons in $\calt(S)$ of type 2.
Let $\bar{S}$ denote the closed surface obtained by attaching a disk to $\partial S$.
Let $\pi \colon \calc(S)\to \calc(\bar{S})$ be the simplicial map associated with the inclusion of $S$ into $\bar{S}$.
The map $\pi$ sends any BP in $S$ to a non-separating curve in $\bar{S}$.

\begin{lem}\label{lem-hex2-pi}
Let $(a, b, c, d, e, f)$ be a 6-tuple defining a hexagon in $\calt(S)$ of type 2 with $a$ and $d$ BP-vertices.
Then the equalities $\pi(b)=\pi(c)$ and $\pi(e)=\pi(f)$ hold, and $\pi(a)$ and $\pi(d)$ are disjoint and distinct.
\end{lem}

\begin{proof}
The first two equalities hold because any of $b$, $c$, $e$ and $f$ are h-vertices, $b$ and $c$ are adjacent, and $e$ and $f$ are adjacent.
Let $A$, $B$, $C$ and $D$ be representatives of $a$, $b$, $c$ and $d$, respectively, such that any two of them intersect minimally.
We identify a curve in $S$ with a curve in $\bar{S}$ through the inclusion of $S$ into $\bar{S}$.
Let $H$ denote the handle cut off by $C$ from $\bar{S}$ and containing $\partial S$.
Let $I$ denote another handle cut off by $C$ from $\bar{S}$.
The BP $A$ lies in the handle cut off by $B$ from $S$ and containing $\partial S$.
This handle contains $I$, and the two curves $B$ and $C$ are isotopic in $\bar{S}$.
It follows that in $\bar{S}$, the two curves in $A$ can be isotoped into curves in $I$.
On the other hand, the BP $D$ lies in $H$.
It turns out that $\pi(a)$ and $\pi(d)$ lie in distinct components of $\bar{S}_{\pi(c)}$.
In particular, $\pi(a)$ and $\pi(d)$ are disjoint and distinct.
\end{proof}

\begin{lem}\label{lem-a0d0}
Let $(a, b, c, d, e, f)$ be a 6-tuple defining a hexagon in $\calt(S)$ of type 2 with $a$ and $d$ BP-vertices.
Then there exist a curve $a_0$ in $a$ and a curve $d_0$ in $d$ such that each of $a_0$ and $d_0$ is disjoint from any of $a,\ldots, f$, and the surface obtained by cutting $S$ along $a_0\cup d_0$ is homeomorphic to $S_{0, 5}$.
\end{lem}

\begin{proof}
Choose representatives $A,\ldots, F$ of $a,\ldots, f$, respectively, such that any two of them intersect minimally. 
Let $R$ denote the component of $S_B$ homeomorphic to $S_{1, 2}$.
Since $A$ is a BP in $R$ and is disjoint from $F$, the intersection $F\cap R$ consists of mutually isotopic, essential simple arcs in $R$ which are non-separating in $R$.
Let $l_F$ be a component of $F\cap R$.
Since $C$ is an h-curve in $R$ and is disjoint from $D$, the intersection $D\cap R$ consists of mutually isotopic, essential simple arcs in $R$ which are separating in $R$.
Let $l_D$ be a component of $D\cap R$.
We find a desired curve $a_0$ in the following two cases individually: (1) There exists a component of $E\cap R$ which is separating in $R$.
(2) There exists no component of $E\cap R$ that is separating in $R$.

In case (1), since $E\cap R$ is disjoint from $D\cap R$, any component of $E\cap R$ that is separating in $R$ is isotopic to $l_D$.
Let $l_E^0$ be a component of $E\cap R$ which is separating in $R$.
Since $E\cap R$ is disjoint from $F\cap R$, the arc $l_E^0$ is disjoint from $l_F$.
As drawn in Figure \ref{fig-arc} (a), we can find the unique component of $A$ disjoint from $l_E^0$.
%====================================
\begin{figure}
\begin{center}
\includegraphics[width=12cm]{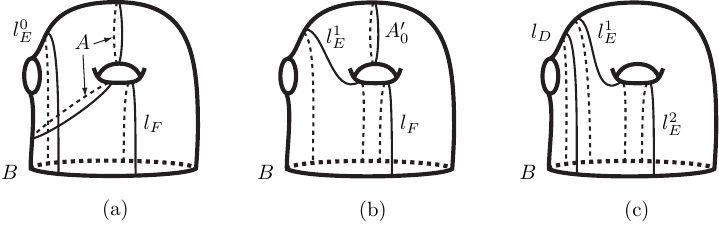}
\caption{}\label{fig-arc}
\end{center}
\end{figure}
%====================================
Let $a_0$ be the isotopy class of that component of $A$.
The curve $a_0$ is disjoint from any of $b$, $d$ and $f$.
Since $C$ is a boundary component of a regular neighborhood of $B\cup D$, the curve $a_0$ is disjoint from $c$.
Similarly, since $E$ is a boundary component of a regular neighborhood of $D\cup F$, the curve $a_0$ is disjoint from $e$.

In case (2), $E\cap R$ consists of essential simple arcs in $R$ which are non-separating in $R$.
Let $\bar{S}$ be the closed surface obtained from $S$ by attaching a disk to $\partial S$.
We identify a curve in $S$ with a curve in $\bar{S}$ through the inclusion of $S$ into $\bar{S}$.
Let $\bar{R}$ denote the component of $\bar{S}_B$ containing $\partial S$.
Since any component of $E\cap R$ is non-separating in $R$, the two curves $B$ and $E$ intersect minimally even as curves in $\bar{S}$, by the criterion on minimal intersection in \cite[Expos\'e 3, Proposition 10]{flp}.
Similarly, $B$ and $F$ also intersect minimally even as curves in $\bar{S}$.
The two curves $E$ and $F$ are isotopic in $\bar{S}$ because they are disjoint h-curves in $S$.
By \cite[Expos\'e 3, Proposition 12]{flp}, there exists a homeomorphism of $\bar{S}$ onto itself isotopic to the identity, fixing $B$ as a set and sending $E\cap R$ to $F\cap R$.
Any component of $E\cap R$ is thus isotopic to $l_F$ in $\bar{R}$.

If any component of $E\cap R$ were isotopic to $l_F$ in $R$, then $e$ and $a$ would be disjoint.
This is a contradiction.
There thus exists a component of $E\cap R$ which is not isotopic to $l_F$ in $R$.
Let $l_E^1$ be such a component of $E\cap R$.

Assuming that there exists no component of $E\cap R$ isotopic to $l_F$ in $R$, we deduce a contradiction.
Any component of $E\cap R$ is then isotopic to $l_E^1$.
Note that if $r_1$ and $r_2$ are non-separating arcs in $R$ which are disjoint and non-isotopic in $R$, but are isotopic in $\bar{R}$, then there exists a homeomorphism of $R$ onto itself sending $r_1$ and $r_2$ to $l_E^1$ and $l_F$, respectively.
It follows that as drawn in Figure \ref{fig-arc} (b), there exists a non-separating curve $A_0'$ in $R$ which is disjoint from $l_F$ and $E\cap R$ and is a boundary component of a regular neighborhood of $l_F\cup B$ in $R$.
This curve $A_0'$ is isotopic to a component of $A$.
There exists a path in $R$ connecting a point of $A_0'$ with a point of $l_F$ without touching $E\cap R$ because any component of $E\cap R$ is isotopic to $l_E^1$.
This contradicts the following:

\begin{claim}
Let $\alpha$ be a BP in $S$.
Let $\beta$ and $\gamma$ be h-curves in $S$ such that each of $\{ \alpha, \beta \}$ and $\{ \beta, \gamma \}$ is an edge of $\calt(S)$.
If a curve $\alpha_0$ in the BP $\alpha$ is disjoint from $\gamma$, then $\alpha_0$ lies in the handle cut off by $\gamma$ from $S$.
In particular, there exists no path in $S$ connecting a point of $\alpha_0$ with a point of $\beta$ without touching $\gamma$.
\end{claim}

\begin{proof}
Let $\alpha_0$ be a curve in the BP $\alpha$ disjoint from $\gamma$.
If $\alpha_0$ were not in the handle cut off by $\gamma$ from $S$, then $\alpha_0$ would be in the handle cut off by $\beta$ from $S$ because $\alpha_0$ is disjoint from $\beta$.
On the other hand, any BP in $S$ disjoint from $\beta$ is in the component of $S_{\beta}$ that is not a handle.
This is a contradiction.
\end{proof}

We have therefore proved that there exists a component of $E\cap R$ isotopic to $l_F$ in $R$.
Let $l_E^2$ be a component of $E\cap R$ isotopic to $l_F$ in $R$.
Cutting $R$ along $l_E^1\cup l_E^2$, we obtain two annuli, one of which contains $\partial S$.
The arc $l_D$ lies in the annulus containing $\partial S$ because $l_D$ is disjoint from $E\cap R$ (see Figure \ref{fig-arc} (c)).
We have the unique component of $A$ isotopic to a curve lying in another annulus.
Let $a_0$ be the isotopy class of that component of $A$.
The curve $a_0$ is disjoint from any of $b$, $d$ and $f$, and is thus disjoint from any of $a,\ldots, f$.

We obtained a curve $a_0$ in $a$ disjoint from any of $a,\ldots, f$ in both cases (1) and (2).
By symmetry, we can also find a curve $d_0$ in $d$ disjoint from any of $a,\ldots, f$.
By Lemma \ref{lem-hex2-pi}, $\pi(a_0)$ and $\pi(d_0)$ lie in distinct components of $\bar{S}_{\pi(b)}$.
It turns out that $a_0$ and $d_0$ are distinct, and the surface obtained by cutting $S$ along $a_0\cup d_0$ is homeomorphic to $S_{0, 5}$. 
\end{proof}

Let $X$ be a surface.
For a BP $b$ in $X$ and a boundary component $\partial$ of $X$, we say that $b$ {\it cuts off} $\partial$ if $b$ cuts off a pair of pants containing $\partial$ from $X$.
For two distinct boundary components $\partial_1$, $\partial_2$ of $X$ and a p-curve $\alpha$ in $X$, we say that $\alpha$ {\it cuts off} $\partial_1$ and $\partial_2$ if $\alpha$ cuts off a pair of pants containing $\partial_1$ and $\partial_2$ from $X$.

\begin{lem}\label{lem-zeta}
Let $(a, b, c, d, e, f)$ be a 6-tuple defining a hexagon in $\calt(S)$ of type 2 with $a$ and $d$ BP-vertices.
Let $a_0$ and $d_0$ be the curves obtained in Lemma \ref{lem-a0d0}.
Then there exists a non-separating curve $\zeta$ in $S$ satisfying the following three conditions:
\begin{enumerate}
\item[(a)] The curve $\zeta$ is disjoint from $a$ and $d$, and belongs to neither $a$ nor $d$.
\item[(b)] Let $S_{d_0}$ denote the surface obtained by cutting $S$ along $d_0$, which is homeomorphic to $S_{1, 3}$.
The pair $\{ a_0, \zeta \}$ is then a BP in $S_{d_0}$, and cuts off one of the two boundary components of $S_{d_0}$ corresponding to $d_0$.
\item[(c)] The condition obtained by exchanging $a_0$ and $d_0$ in condition (b) holds.
\end{enumerate}
Moreover, such a curve $\zeta$ uniquely exists up to isotopy.
\end{lem}

Let $\Pi$ be a hexagon in $\calt(S)$ of type 2, and let $(a, b, c, d, e, f)$ be a 6-tuple defining $\Pi$ with $a$ and $d$ BP-vertices.
We denote by $\zeta(\Pi)$ the curve $\zeta$ obtained by applying Lemma \ref{lem-zeta} to $\Pi$.
Let $d_1$ be the curve in $d$ distinct from $d_0$.
In the surface $S_{d_0}$, $a$ and $\{ a_0, \zeta \}$ are BPs, $b$ and $d_1$ are p-curves, and $c$ is an h-curve.
The 5-tuple $(a, b, c, d_1, \{ a_0, \zeta \})$ defines a pentagon in $\calt(S_{d_0})$.

\begin{proof}[Proof of Lemma \ref{lem-zeta}]
Choose representatives $A,\ldots, F$ of $a,\ldots, f$, respectively, such that any two of them intersect minimally.
Let $A_0$ and $A_1$ denote the two components of $A$ so that the isotopy class of $A_0$ is $a_0$.
Let $D_0$ and $D_1$ denote the two components of $D$ so that the isotopy class of $D_0$ is $d_0$.
We define $T$ as the surface obtained by cutting $S$ along $A_0\cup D_0$, which is homeomorphic to $S_{0, 5}$.
We label boundary components of $T$ as $\partial$, $\partial_a^1$, $\partial_a^2$, $\partial_d^1$ and $\partial_d^2$ so that $\partial$ corresponds to $\partial S$, $\partial_a^1$ and $\partial_a^2$ correspond to $A_0$, and $\partial_d^1$ and $\partial_d^2$ correspond to $D_0$.
Without loss of generality, we may assume that $A_1$ is a p-curve in $T$ cutting off $\partial$ and $\partial_a^2$.
Each of $B$ and $F$ is a p-curve in $T$ cutting off $\partial_d^1$ and $\partial_d^2$.
Similarly, each of $C$ and $E$ is a p-curve in $T$ cutting off $\partial_a^1$ and $\partial_a^2$.

Let $R$ be the component of $T_{A_1}$ homeomorphic to $S_{0, 4}$.
The surface $R$ contains $\partial_a^1$, $\partial_d^1$ and $\partial_d^2$.
For each essential simple arc $l$ in $R$ whose boundary lies in $A_1$ and for each $\partial_j^k\in \{ \partial_a^1, \partial_d^1, \partial_d^2\}$, we say that $l$ {\it cuts off} $\partial_j^k$ if $\partial_j^k$ lies in the annulus cut off by $l$ from $R$ (see Figure \ref{fig-arcs} (a)).
Since $B$ is a curve in $R$ and is disjoint from $C$, the intersection $C\cap R$ consists of mutually isotopic, essential simple arcs in $R$ cutting off $\partial_a^1$.
Similarly, since $F$ is a curve in $R$ and is disjoint from $E$, the intersection $E\cap R$ consists of mutually isotopic, essential simple arcs in $R$ cutting off $\partial_a^1$.
Pick a component $l_C$ of $C\cap R$ and a component $l_E$ of $E\cap R$.

\begin{claim}\label{claim-ce}
The two arcs $l_C$ and $l_E$ are non-isotopic, and cannot be isotoped so that they are disjoint. 
\end{claim}

\begin{proof}
The former assertion holds because otherwise $B$ and $F$ would be isotopic.
The latter assertion holds because $l_C$ and $l_E$ are non-isotopic and because both $l_C$ and $l_E$ cut off $\partial_a^1$.
\end{proof}

\begin{claim}\label{claim-d1}
The intersection $D_1\cap R$ consists of mutually isotopic, essential simple arcs in $R$.
\end{claim}

\begin{proof}
Assuming the contrary, we deduce a contradiction.
Any family of essential simple arcs in $R$ which are mutually disjoint and non-isotopic and whose boundaries lie in $A_1$ has at most three elements.
If $D_1\cap R$ had three components which are mutually non-isotopic, then $l_C$ and $l_E$ would be isotopic because $l_C$ and $l_E$ are disjoint from $D_1\cap R$.
This contradicts Claim \ref{claim-ce}.
We thus assume that $D_1\cap R$ contains exactly two essential simple arcs in $R$ up to isotopy.
Let $l_D^1$ and $l_D^2$ be components of $D_1\cap R$ which are non-isotopic.

If either $l_D^1$ or $l_D^2$ cut off $\partial_a^1$, then $l_C$ and $l_E$ would be isotopic to that arc.
This also contradicts Claim \ref{claim-ce}.
It follows that one of $l_D^1$ and $l_D^2$ cuts off $\partial_d^1$ and another cuts off $\partial_d^2$.
Without loss of generality, we may assume that $l_D^1$ cuts off $\partial_d^1$, and $l_D^2$ cuts off $\partial_d^2$.
%====================================
\begin{figure}
\begin{center}
\includegraphics[width=10cm]{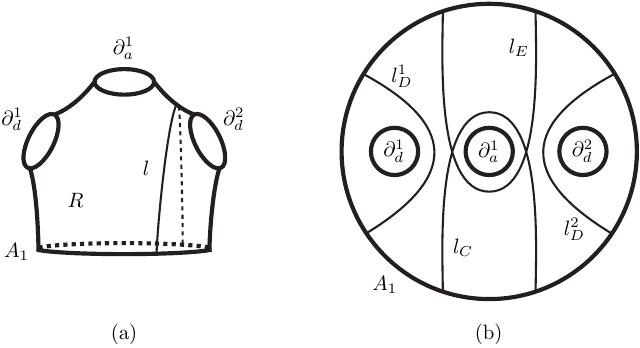}
\caption{The arc $l$ in (a) cuts off $\partial_d^2$.}\label{fig-arcs}
\end{center}
\end{figure}
%====================================
Claim \ref{claim-ce} implies that $l_C$ and $l_E$ are drawn as in Figure \ref{fig-arcs} (b).
For each $k=1, 2$, there exists a path in $R$ connecting a point of $\partial_d^k$ with a point of $l_D^k$ without touching neither $C\cap R$ nor $E\cap R$.

The curve $D_1$ is a p-curve in $T$ cutting off $\partial$ and one of $\partial_d^1$ and $\partial_d^2$.
Suppose that $D_1$ cuts off $\partial$ and $\partial_d^2$.
We define $U$ as the surface obtained from $T$ by attaching a disk to $\partial_d^1$.
The two curves $C$ and $D_1$ are isotopic in $U$ because $C$ and $D_1$ are disjoint and the pair of pants cut off from $T$ by each of them does not contain $\partial_d^1$.
Similarly, $D_1$ and $E$ are also isotopic in $U$.
It turns out that $C$ and $E$ are isotopic in $U$.
On the other hand, $C$ and $E$ intersect minimally as curves in $T$, and $C\cap E$ is non-empty.
By \cite[Expos\'e 3, Proposition 10]{flp}, there exist a subarc in $C$ and a subarc in $E$ whose union is a simple closed curve in $T$ isotopic to $\partial_d^1$.
The curve $D_1$ is disjoint from $C$ and $E$.
Any path in $T$ connecting a point of $\partial_d^1$ with a point of $D_1$ therefore intersects either $C$ or $E$.
This contradicts the property obtained in the end of the last paragraph.
Exchanging $\partial_d^1$ and $\partial_d^2$, we can deduce a contradiction if we suppose that $D_1$ cuts off $\partial$ and $\partial_d^1$.
\end{proof}

By Claim \ref{claim-d1}, there exists an essential simple closed curve in $R$ disjoint from $D_1\cap R$, which is unique up to isotopy.
Let $\zeta$ denote the isotopy class of that curve.
This is a desired one.
In fact, condition (a) holds by definition.
Claim \ref{claim-ce} implies that any component of $D_1\cap R$ cuts off either $\partial_d^1$ or $\partial_d^2$.
The curve $\zeta$ is therefore a p-curve in $R$ cutting off $\partial_a^1$ and one of $\partial_d^1$ and $\partial_d^2$.
Conditions (b) and (c) follow.
The uniqueness of $\zeta$ holds because there exists at most one curve in $T$ disjoint from the two curves $a_1$ and $d_1$ that intersect.
\end{proof}

In the proof of the subsequent two theorems, we use the following:

\medskip

\noindent {\bf Graph $\calf$.} Let $R$ be a surface homeomorphic to $S_{0, 4}$.
We define a simplicial graph $\calf =\calf(R)$ so that the set of vertices of $\calf$ is $V(R)$, and two vertices $\alpha$, $\beta$ of $\calf$ are connected by an edge of $\calf$ if and only if $i(\alpha, \beta)=2$.

\medskip

It is well known that this graph is isomorphic to the Farey graph realized as an ideal triangulation of the Poincar\'e disk (see \cite[Section 3.2]{luo} or Figure \ref{fig-farey} (a)).
We mean by a {\it triangle} in $\calf$ a subgraph of $\calf$ consisting of exactly three vertices and exactly three edges.
Note that for any two ordered triples of vertices in $\calf$ defining triangles in $\calf$, there exists a unique simplicial automorphism of $\calf$ sending the first triple to the second one. 

The following theorem characterizes hexagons in $\calt(S)$ of type 2.

\begin{thm}\label{thm-hex2-m}
Let $\Pi$ be a hexagon in $\calt(S)$ of type 2, and let $(a, b, c, d, e, f)$ be a 6-tuple defining $\Pi$ with $a$ and $d$ BP-vertices.
Put $\zeta =\zeta(\Pi)$.
Then there exists a unique non-zero integer $m$ with $f=t_{\zeta}^m(b)$ and $e=t_{\zeta}^m(c)$.
\end{thm}

\begin{proof}
Let $a_0$ and $d_0$ be the curves in the BPs $a$ and $d$, respectively, obtained in Lemma \ref{lem-a0d0}.
The surface $S_{d_0}$ is homeomorphic to $S_{1, 3}$.
In $S_{d_0}$, the curve in $d$ distinct from $d_0$, denoted by $d_1$, is a p-curve, $b$ is a p-curve, $c$ is an h-curve, and $a$ and $\{ a_0, \zeta \}$ are BPs.
The 5-tuple $(a, b, c, d_1, \{ a_0, \zeta \})$ defines a pentagon in $\calt(S_{d_0})$ (see Figure \ref{fig-pen-pf}).
%====================================
\begin{figure}
\begin{center}
\includegraphics[width=5cm]{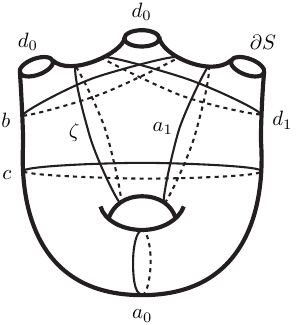}
\caption{}\label{fig-pen-pf}
\end{center}
\end{figure}
%====================================
Similarly, the 5-tuple $(a, f, e, d_1, \{ a_0, \zeta \})$ also defines a pentagon in $\calt(S_{d_0})$ such that in $S_{d_0}$, $e$ is an h-curve and $f$ is a p-curve. 
Cut $S_{d_0}$ along $a$.
The obtained surface consists of a pair of pants and a surface homeomorphic to $S_{0, 4}$.
Let $R$ denote the latter component.

Let $\partial_d^1$ and $\partial_d^2$ denote the two boundary components of $R$ corresponding to $d_0$.
The curves $b$ and $f$ lie in $R$ and cut off $\partial_d^1$ and $\partial_d^2$.
The curve $\zeta$ cuts off a pair of pants from $R$ containing exactly one of $\partial_d^1$ and $\partial_d^2$.
By Lemma \ref{lem-hex2-pen} (ii), we have $i(\zeta, b)=i(\zeta, f)=2$.
Looking at the action of the Dehn twist $t_{\zeta}$ on the graph $\calf(R)$, we see that $t_{\zeta}$ acts on the link of $\zeta$ in $\calf(R)$ freely.
Moreover, $t_{\zeta}$ transitively acts on the set of all vertices in the link of $\zeta$ that correspond to curves in $R$ cutting off $\partial_d^1$ and $\partial_d^2$.
It follows that there exists a unique integer $m$ with $t_{\zeta}^m(b)=f$.
Since $b$ and $f$ are distinct, the integer $m$ is non-zero.

The 6-tuple $(a, b, c, d, t_{\zeta}^m(c), t_{\zeta}^m(b))$ defines a hexagon in $\calt(S)$, as shown in the beginning of this section.
There exists at most one h-curve in $S$ disjoint from the BP $d$ and the h-curve $t_{\zeta}^m(b)=f$.
We therefore have $t_{\zeta}^m(c)=e$.
\end{proof}

\begin{thm}\label{thm-hex2-3path}
Let $\Pi$ be a hexagon in $\calt(S)$ of type 2, and let $(a, b, c, d, e, f)$ be a 6-tuple defining $\Pi$ with $a$ and $d$ BP-vertices.
Put $\zeta =\zeta(\Pi)$.
Then the following assertions hold:
\begin{enumerate}
\item If neither $f=t_{\zeta}(b)$ nor $f=t_{\zeta}^{-1}(b)$, then $\Pi$ is the only hexagon in $\calt(S)$ of type 2 containing $f$, $a$, $b$ and $c$.
\item If either $f=t_{\zeta}(b)$ or $f=t_{\zeta}^{-1}(b)$, then there exists exactly one hexagon in $\calt(S)$ of type 2 that is distinct from $\Pi$ and contains $f$, $a$, $b$ and $c$.
\end{enumerate}
\end{thm}

Before proving Theorem \ref{thm-hex2-3path}, we prepare two lemmas.

\begin{lem}\label{lem-zeta-eq}
Let $\Pi$ be a hexagon in $\calt(S)$ of type 2, and let $(a, b, c, d, e, f)$ be a 6-tuple defining $\Pi$ with $a$ and $d$ BP-vertices.
Let $\Omega$ be a hexagon in $\calt(S)$ of type 2 containing $f$, $a$, $b$ and $c$.
If $\zeta(\Pi)=\zeta(\Omega)$, then $\Pi =\Omega$.
\end{lem}

\begin{proof}
Put $\zeta =\zeta(\Pi)=\zeta(\Omega)$.
By Theorem \ref{thm-hex2-m}, there exists a unique integer $m$ with $t_{\zeta}^m(b)=f$ and $t_{\zeta}^m(c)=e$.
Let $(a, b, c, d', e', f)$ be the 6-tuple defining $\Omega$.
Applying the same theorem to $\Omega$, we obtain a unique integer $n$ with $t_{\zeta}^n(b)=f$ and $t_{\zeta}^n(c)=e'$.
The equality $t_{\zeta}^m(b)=t_{\zeta}^n(b)$ then holds.
We thus have $m=n$ and $e=e'$.
Since at most one BP in $S$ disjoint from $c$ and $e$ exists, we have $d=d'$.
\end{proof}

We set $R=S_{0, 4}$.
For any edge $\tau$ of the graph $\calf=\calf(R)$, the complement of $\tau \cup \partial \tau$ in the geometric realization of $\calf$ has exactly two connected components.
We call those components {\it sides} of $\tau$.

\begin{lem}\label{lem-04}
We set $R=S_{0, 4}$ and $\calf=\calf(R)$.
Let $\alpha$ and $\beta$ be curves in $R$ with $i(\alpha, \beta)=2$.
We denote by $\gamma$ the only curve in $R$ such that each of $\{ \alpha, \beta, \gamma \}$ and $\{ \alpha, \beta, t_{\alpha}(\gamma)\}$ defines a triangle in $\calf$.
Let $\delta$ be a curve in $R$ with $\delta \neq \alpha$ and $i(\beta, \delta)=2$.
Let $m$ and $n$ be non-zero integers.
If the equality $t_{\alpha}^m(\beta)=t_{\delta}^n(\beta)$ holds, then either $\delta =\gamma$ and $(m, n)=(-1, 1)$ or $\delta =t_{\alpha}(\gamma)$ and $(m, n)=(1, -1)$.
\end{lem}

\begin{proof}
Realize the graph $\calf$ geometrically as an ideal triangulation of the Poincar\'e disk $D$.
The set $\partial D\setminus \{ \alpha, \beta, \gamma, t_{\alpha}(\gamma)\}$ consists of the four connected components $L_1$, $L_2$, $L_3$ and $L_4$ as in Figure \ref{fig-farey} (b).
%====================================
\begin{figure}
\begin{center}
\includegraphics[width=11cm]{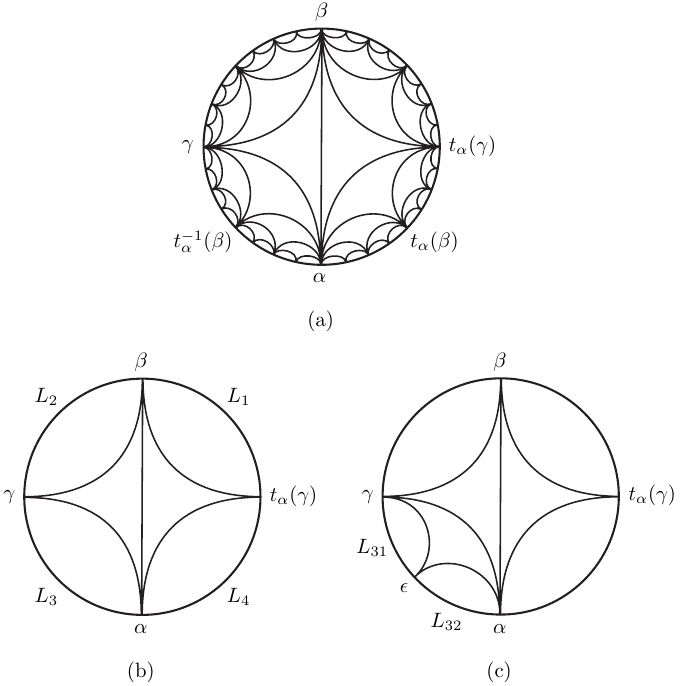}
\caption{}\label{fig-farey}
\end{center}
\end{figure}
%====================================
For any positive integer $j$, $t_{\alpha}^j(\beta)$ lies in $L_4$.
For any negative integer $k$, $t_{\alpha}^k(\beta)$ lies in $L_3$.

The vertex $\delta$ is in the link of $\beta$ in $\calf$ and distinct from $\alpha$.
Assuming that $\delta$ is equal to neither $\gamma$ nor $t_{\alpha}(\gamma)$, we deduce a contradiction.
The vertex $\delta$ then lies in either $L_1$ or $L_2$.
We have the two triangles in $\calf$ containing the edge $\{ \beta, \delta \}$.
Each of those triangles has the edge containing $\delta$ and distinct from $\{ \beta, \delta \}$.
Let $\tau$ and $\sigma$ denote those edges.
If $\delta$ lies in $L_1$, then the interior of $\tau$ and that of $\sigma$ lie in the side of the edge $\{ \beta, t_{\alpha}(\gamma)\}$ containing $\delta$.
The argument in the previous paragraph shows that for any non-zero integer $j$, $t_{\delta}^j(\beta)$ lies in $L_1$.
This contradicts the equality $t_{\alpha}^m(\beta)=t_{\delta}^n(\beta)$.
We can deduce a contradiction similarly if we assume that $\delta$ lies in $L_2$.
It turns out that $\delta$ is equal to either $\gamma$ or $t_{\alpha}(\gamma)$.

We first suppose the equality $\delta =\gamma$.
Let $\epsilon$ denote the vertex $t_{\alpha}^{-1}(\beta)=t_{\gamma}(\beta)$, which lies in $L_3$ and forms a triangle in $\calf$ together with $\alpha$ and $\gamma$.
Let $L_{31}$ and $L_{32}$ be the two components of $L_3\setminus \{ \epsilon \}$ so that the closure of $L_{31}$ contains $\gamma$ and that of $L_{32}$ contains $\alpha$ (see Figure \ref{fig-farey} (c)).
For any integer $j$ with $j>1$, $t_{\gamma}^j(\beta)$ lies in $L_{31}$.
For any integer $k$ with $k<-1$, $t_{\alpha}^k(\beta)$ lies in $L_{32}$.
For any negative integer $k$, $t_{\gamma}^k(\beta)$ lies in $L_2$.
The equality $t_{\alpha}^m(\beta)=t_{\gamma}^n(\beta)$ therefore implies $(m, n)=(-1, 1)$.
If we suppose the equality $\delta =t_{\alpha}(\gamma)$ in place of the equality $\delta =\gamma$, then we obtain $(m, n)=(1, -1)$ along a similar argument.
\end{proof}

We are now ready to prove Theorem \ref{thm-hex2-3path}.

\begin{proof}[Proof of Theorem \ref{thm-hex2-3path} (i)]
Let $\Pi$ be a hexagon in $\calt(S)$ of type 2.
Let $(a, b, c, d, e, f)$ be a 6-tuple defining $\Pi$ with $a$ and $d$ BP-vertices.
Pick a hexagon $\Omega$ in $\calt(S)$ of type 2 containing $f$, $a$, $b$ and $c$.
Let $(a, b, c, d', e', f)$ be the 6-tuple defining $\Omega$.
We put $\zeta =\zeta(\Pi)$ and $\eta =\zeta(\Omega)$.
By Theorem \ref{thm-hex2-m}, we have the non-zero integers $m$, $n$ with
\[f=t_{\zeta}^m(b)=t_{\eta}^n(b),\quad e=t_{\zeta}^m(c)\quad \textrm{and}\quad e'=t_{\eta}^n(c).\]
Applying Lemma \ref{lem-a0d0} to $\Pi$, we obtain the curve $a_0$ in $a$ and the curve $d_0$ in $d$ that are disjoint from any of $a,\ldots, f$.
In the component of $S_a$ homeomorphic to $S_{1, 2}$, the curve $d_0$ is the only curve disjoint from $b$ and $f$.
Applying Lemma \ref{lem-a0d0} to the hexagon $\Omega$, which contains $a$, $b$ and $f$, we see that $d_0$ is also contained in the BP $d'$.
Let $R$ denote the subsurface of $S$ filled by $b$ and $f$, which is homeomorphic to $S_{0, 4}$.
Any of $b$, $f$, $\zeta$ and $\eta$ is a curve in $R$.
By Lemmas \ref{lem-hex2-pen} and \ref{lem-zeta}, we have $i(b, \zeta)=i(b, \eta)=2$.

If $\zeta$ and $\eta$ are distinct, then by Lemma \ref{lem-04}, the equality $t_{\zeta}^m(b)=t_{\eta}^n(b)$ implies that either $(m, n)=(1, -1)$ or $(m, n)=(-1, 1)$.
It follows that either $f=t_{\zeta}(b)$ or $f=t_{\zeta}^{-1}(b)$.
Under the assumption that neither $f=t_{\zeta}(b)$ nor $f=t_{\zeta}^{-1}(b)$ holds, we therefore have the equality $\zeta =\eta$.
By Lemma \ref{lem-zeta-eq}, we then have $\Pi =\Omega$.
Theorem \ref{thm-hex2-3path} (i) is proved.
\end{proof}

\begin{proof}[Proof of Theorem \ref{thm-hex2-3path} (ii)]
Let $\Pi$ be a hexagon in $\calt(S)$ of type 2.
Let $(a, b, c, d, e, f)$ be a 6-tuple defining $\Pi$ with $a$ and $d$ BP-vertices.
Put $\zeta =\zeta(\Pi)$.
Let $a_0$ and $d_0$ be the curves in the BPs $a$ and $d$, respectively, obtained in Lemma \ref{lem-a0d0}.
We define curves $a_1$ and $d_1$ so that $a=\{ a_0, a_1\}$ and $d=\{ d_0, d_1\}$.
Let $R$ denote the subsurface of $S$ filled by $b$ and $f$, which is homeomorphic to $S_{0, 4}$ because $b$ and $f$ are disjoint from $a$ and $d_0$.
We set $\calf=\calf(R)$.
Any of $b$, $f$ and $\zeta$ is a curve in $R$.
%====================================
\begin{figure}
\begin{center}
\includegraphics[width=11cm]{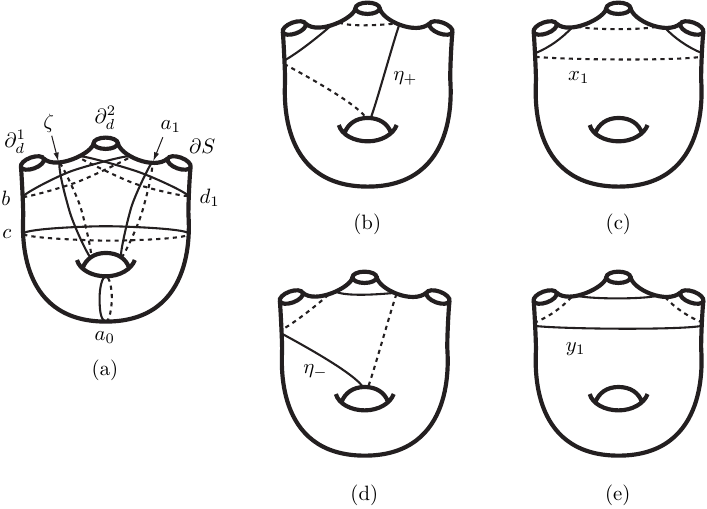}
\caption{}\label{fig-hex2}
\end{center}
\end{figure}
%====================================
By Lemmas \ref{lem-hex2-pen} and \ref{lem-zeta}, the curves $a_0$, $a_1$, $b$, $c$, $d_1$ and $\zeta$ in $S_{d_0}$ are drawn as in Figure \ref{fig-hex2} (a), where $\partial_d^1$ and $\partial_d^2$ denote the two boundary components of $S_{d_0}$ corresponding to $d_0$.

We first suppose the equality $f=t_{\zeta}(b)$.
We construct a hexagon in $\calt(S)$ of type 2 containing $f$, $a$, $b$ and $c$ and distinct from $\Pi$.
The assumption $f=t_{\zeta}(b)$ implies that there exists a unique curve $\eta_+$ in $R$ such that each of the triples $\{ b, \zeta, \eta_+ \}$ and $\{ f, \zeta, \eta_+\}$ forms a triangle in $\calf$, as in Figure \ref{fig-tri} (a).
%====================================
\begin{figure}
\begin{center}
\includegraphics[width=11cm]{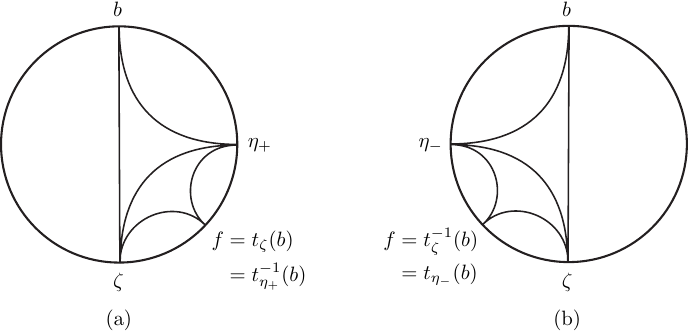}
\caption{}\label{fig-tri}
\end{center}
\end{figure}
%====================================
We have the equality $f=t_{\zeta}(b)=t_{\eta_+}^{-1}(b)$.
The curve $\eta_+$ is then determined as in Figure \ref{fig-hex2} (b). 
We define $x_1$ as the curve drawn in Figure \ref{fig-hex2} (c), and set $x=\{ d_0, x_1\}$.
The 5-tuple $(a, b, c, x_1, \{ a_0, \eta_+\})$ defines a pentagon in $\calt(S_{d_0})$. 
The 6-tuple $(a, b, c, x, t_{\eta_+}^{-1}(c), f)$ therefore defines a hexagon in $\calt(S)$ of type 2, denoted by $\Omega_+$.

Let $\Omega$ be a hexagon in $\calt(S)$ of type 2 containing $f$, $a$, $b$ and $c$.
Put $\eta =\zeta(\Omega)$.
Applying Theorem \ref{thm-hex2-m} to $\Omega$, we have a non-zero integer $n$ with $f=t_{\eta}^n(b)$.
In the first paragraph in the proof of Theorem \ref{thm-hex2-3path} (i), we showed that $\eta$ is also a curve in $R$, and we have $i(b, \zeta)=i(b, \eta)=2$.
The equality $f=t_{\zeta}(b)=t_{\eta}^n(b)$ and Lemma \ref{lem-04} imply that either $\eta =\zeta$ or $\eta =\eta_+$ and $n=-1$.
By Lemma \ref{lem-zeta-eq}, we have either $\Omega =\Pi$ or $\Omega =\Omega_+$.
Theorem \ref{thm-hex2-3path} (ii) is therefore proved if $f=t_{\zeta}(b)$.

We next suppose the equality $f=t_{\zeta}^{-1}(b)$.
There exists a unique curve $\eta_-$ in $R$ such that each of the triples $\{ b, \zeta, \eta_- \}$ and $\{ f, \zeta, \eta_-\}$ forms a triangle in $\calf$, as in Figure \ref{fig-tri} (b).
We have the equality $f=t_{\zeta}^{-1}(b)=t_{\eta_-}(b)$.
The curve $\eta_-$ is then determined as in Figure \ref{fig-hex2} (d).
We define a curve $y_1$ as in Figure \ref{fig-hex2} (e), and set $y=\{ d_0, y_1\}$.
The 5-tuple $(a, b, c, y_1, \{ a_0, \eta_-\})$ defines a pentagon in $\calt(S_{d_0})$. 
The 6-tuple $(a, b, c, y, t_{\eta_-}(c), f)$ defines a hexagon in $\calt(S)$ of type 2, denoted by $\Omega_-$.
As in the previous paragraph, we can show that if $\Omega$ is a hexagon in $\calt(S)$ of type 2 containing $f$, $a$, $b$ and $c$, then either $\Omega =\Pi$ or $\Omega =\Omega_-$.
\end{proof}

In the rest of this section, we observe hexagons in $\calt(S)$ sharing a 3-path with a given hexagon of type 1 or type 2.
Note that a hexagon in $\calt(S)$ has exactly two BP-vertices if and only if it is of either type 1 or type 2.

\begin{lem}\label{lem-hex1-hex}
Let $\Pi$ be a hexagon in $\calt(S)$ of type 1, and let $(a, b, c, d, e, f)$ be a 6-tuple defining $\Pi$ with $b$ and $f$ BP-vertices.
Let $\Omega$ be a hexagon in $\calt(S)$ such that $\Pi \cap \Omega$ contains a 3-path.
Then the following assertions hold:
\begin{enumerate}
\item The hexagon $\Omega$ contains at least one of $b$ and $f$.
\item If $\Omega$ contains exactly one of $b$ and $f$, then $\Omega$ has exactly two BP-vertices.
\item If $\Omega$ contains both $b$ and $f$, then the equality $\Omega =\Pi$ holds.
\end{enumerate}
\end{lem}

\begin{proof}
Assertion (i) holds because any 3-path in $\Pi$ contains at least one of $b$ and $f$. 
If $\Omega$ contains exactly one of $b$ and $f$, then $\Omega$ contains two adjacent h-vertices, and thus has exactly two BP-vertices by Lemmas \ref{6-h_hex} and \ref{one-BP_hex}. 
Assertion (ii) follows.

Assuming that $\Omega$ contains $b$ and $f$, we prove assertion (iii).
Without loss of generality, we may assume that $\Pi$ and $\Omega$ contain $f$, $a$, $b$ and $c$.
Let $(a, b, c, d', e', f)$ be the 6-tuple defining $\Omega$.
The vertex $e'$ is an h-vertex because $f$ is a BP-vertex.
Assuming that $d'$ is a BP-vertex, we deduce a contradiction.
Let $\alpha$ be the curve contained in $b$ and $f$.

Choose representatives $A,\ldots, F$, $D'$ and $E'$ of $a,\ldots, f$, $d'$ and $e'$, respectively, such that any two of them intersect minimally.
Let $\mathfrak{a}$ denote the component of $F$ whose isotopy class is $\alpha$.
Let $R$ denote the component of $S_C$ that is not a handle. 
Note that $\mathfrak{a}$ is a curve in $R$.
Since $D'$ is a BP in $R$, the intersection $E'\cap R$ consists of mutually isotopic, essential simple arcs in $R$ which are non-separating in $R$ (see Figure \ref{fig-prime}).
%====================================
\begin{figure}
\begin{center}
\includegraphics[width=4cm]{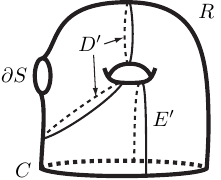}
\caption{}\label{fig-prime}
\end{center}
\end{figure}
%==================================== 
It follows from $E'\cap F=\emptyset$ that $E'\cap R$ is disjoint from $\mathfrak{a}$. 
Since the two components of $D'$ are boundary components of a regular neighborhood of $(E'\cap R) \cup C$ in $R$, one of components of $D'$ is isotopic to $\mathfrak{a}$.
It turns out that $d'$ contains $\alpha$.

We define curves $b_1$, $d_1'$ and $f_1$ so that $b=\{\alpha, b_1\}$, $d'=\{\alpha, d_1'\}$ and $f=\{\alpha, f_1\}$. 
The 6-tuple $(a, b_1, c, d_1', e, f_1)$ then defines a hexagon in $\calc_s(S_{\alpha})$ such that each of the curves $b_1$, $d_1'$ and $f_1$ in $S_{\alpha}$ cuts off a pair of pants containing $\partial S$ from $S_{\alpha}$.
This is a contradiction because by Theorem \ref{thm-13}, for any hexagon $H$ in $\calc_s(S_{\alpha})$, there is a p-curve in $H$ cutting off a pair of pants containing the two boundary components of $S_{\alpha}$ that correspond to $\alpha$, from $S_{\alpha}$.

We proved that $d'$ is an h-vertex.
It follows that $\Omega$ is of type 1.
By Lemma \ref{lem-hex1-hex1}, we have the equality $\Omega =\Pi$.
\end{proof}

Finally, we obtain the following:

\begin{thm}\label{thm-hex1-hex2}
Let $\Pi$ be a hexagon in $\calt(S)$.
Then the following assertions hold:
\begin{enumerate}
\item If $\Pi$ is of type 1, then for any 3-path $K$ in $\Pi$ containing the two BP-vertices of $\Pi$, there exists no hexagon in $\calt(S)$ distinct from $\Pi$ and containing $K$.
\item If $\Pi$ is of type 2, then for any 3-path $L$ in $\Pi$ containing exactly one BP-vertex of $\Pi$, there exist at most two hexagons in $\calt(S)$ distinct from $\Pi$ and containing $L$.
\end{enumerate}
\end{thm}

\begin{proof}
Assertion (i) follows from Lemma \ref{lem-hex1-hex} (iii).
Suppose that $\Pi$ is of type 2, and pick a 3-path $L$ in $\Pi$ containing exactly one BP-vertex of $\Pi$. 
By Lemmas \ref{6-h_hex} and \ref{one-BP_hex}, any hexagon in $\calt(S)$ has at least two BP-vertices.
Any hexagon in $\calt(S)$ containing $L$ is thus of either type 1 or type 2 because $L$ contains two adjacent h-vertices.
By Lemma \ref{lem-hex1-hex1}, the number of hexagons in $\calt(S)$ of type 1 containing $L$ is at most 1.
By Theorem \ref{thm-hex2-3path}, the number of hexagons in $\calt(S)$ of type 2 distinct from $\Pi$ and containing $L$ is at most 1.
Assertion (ii) is therefore proved.
\end{proof}

\begin{rem}
In addition to Theorem \ref{thm-hex1-hex2}, we have the following description of the number of hexagons sharing a 3-path with a given hexagon of type 1 or type 2, whose proof is not presented here because it is not used in the rest of the paper.

Let $\Pi$ be a hexagon in $\calt(S)$ defined by a 6-tuple $(a, b, c, d, e, f)$.
Assume that $\Pi$ is of type 1 with $b$ and $f$ BP-vertices.
Let $K$ be a 3-path in $\Pi$ containing exactly one of $b$ and $f$.
If $K$ does not contain $a$, then $\Pi$ is the only hexagon in $\calt(S)$ containing $K$ by Lemma \ref{lem-hex1-hex1}.
If $K$ contains $a$, then any hexagon in $\calt(S)$ distinct from $\Pi$ and containing $K$ is of type 2 by Lemma \ref{lem-hex1-hex1}, and there exist exactly two hexagons in $\calt(S)$ of type 2 containing $K$.

Those two hexagons are drawn in Figure \ref{fig-hex-rem}.
%====================================
\begin{figure}
\begin{center}
\includegraphics[width=12cm]{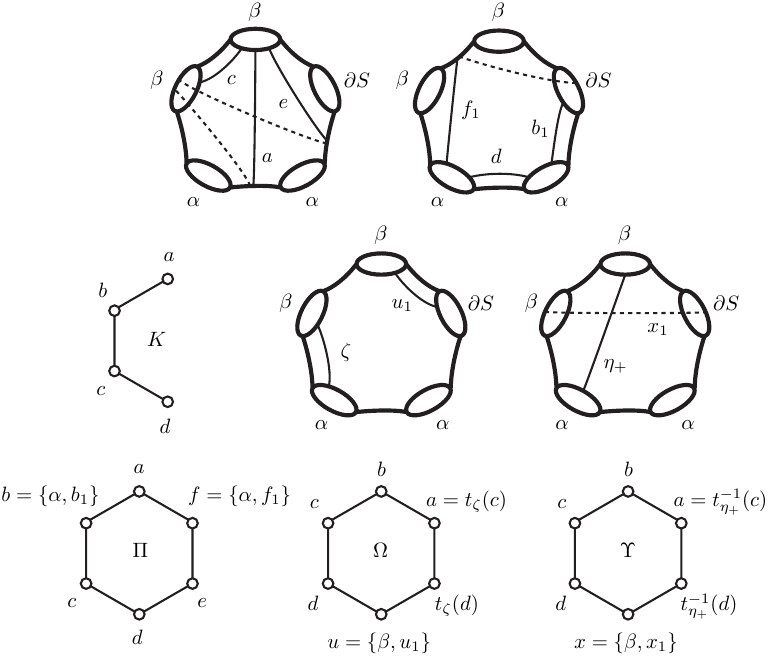}
\caption{}\label{fig-hex-rem}
\end{center}
\end{figure}
%==================================== 
Let $\alpha$ and $\beta$ be disjoint and non-isotopic curves in $S$ such that the surface obtained by cutting $S$ along $\alpha \cup \beta$, denoted by $T$, is connected.
Any essential simple arc $l$ in $T$ connecting two distinct boundary components $\partial_1$, $\partial_2$ associates a curve $c(l)$ in $T$.
Namely, $c(l)$ is defined as a boundary component of a regular neighborhood of $l\cup \partial_1\cup \partial_2$ in $T$.
In Figure \ref{fig-hex-rem}, the surface $T$ is drawn, and in place of curves, essential simple arcs associating them are drawn.
Given a hexagon $\Pi$ in $\calt(S)$ of type 1 and a 3-path $K$ in $\Pi$ as drawn in Figure \ref{fig-hex-rem}, we have the two hexagons $\Omega$, $\Upsilon$ in $\calt(S)$ of type 2 containing $K$.
It follows from Theorem \ref{thm-hex1-hex2} (ii) that there is no other hexagon in $\calt(S)$ containing $K$.

We next assume that $\Pi$ is of type 2 with $a$ and $d$ BP-vertices.
Let $\zeta$ be the curve $\zeta(\Pi)$ obtained in Lemma \ref{lem-zeta}.
Let $L$ be a 3-path in $\Pi$.
Any hexagon in $\calt(S)$ containing $L$ is either of type 1 or type 2 because $L$ contains two adjacent h-vertices.
We first suppose that $L$ contains exactly one of $a$ and $d$.
If either $f=t_{\zeta}(b)$ or $f=t_{\zeta}^{-1}(b)$, then there exist exactly two hexagons in $\calt(S)$ distinct from $\Pi$ and containing $L$, one of which is of type 1 and another of which is of type 2.
We omit to describe those two hexagons because they are obtained by using Figure \ref{fig-hex-rem} after exchanging symbols appropriately. 
If neither $f=t_{\zeta}(b)$ nor $f=t_{\zeta}^{-1}(b)$, then $\Pi$ is the only hexagon in $\calt(S)$ containing $L$.
Finally, we suppose that $L$ contains $a$ and $d$.
In the fourth paragraph of this section, we have observed that there are infinitely many hexagons in $\calt(S)$ of type 2 containing $L$.
\end{rem}

%%%%%%%%%%%%%%%%%%%%%%%%%%%%%%%%%%%%%%%%

\section{Hexagons of type 3}\label{sec-type3}

Throughout this section, we set $S=S_{2,1}$.
We say that a hexagon in $\calt(S)$ is of {\it type 3} if it contains exactly three BP-vertices.
In this section, we focus on the hexagons of type 3 drawn in Figure \ref{fig-hex3}, and present their property that no hexagon of type 1 or type 2 satisfies.
%====================================
\begin{figure}
\begin{center}
\includegraphics[width=10cm]{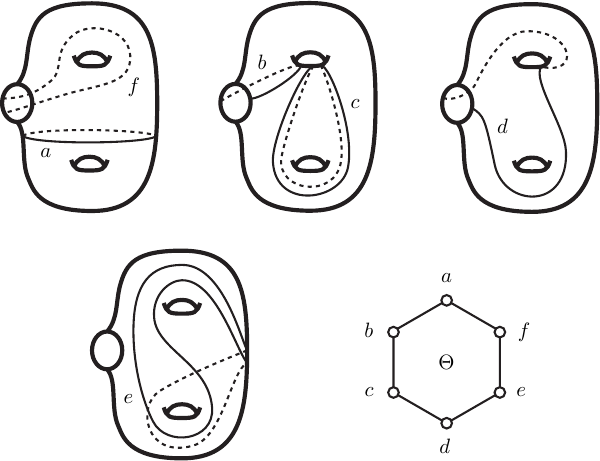}
\caption{}\label{fig-hex3}
\end{center}
\end{figure}
%====================================

We note that there is a one-to-one correspondence between elements of $V_{bp}(S)$ and elements of $V_a(S)$ whose representatives are non-separating in $S$.
In fact, each BP $b$ in $S$ associates an essential simple arc in $S$ contained in the pair of pants cut off by $b$ from $S$, which is non-separating in $S$ and is uniquely determined up to isotopy.
Conversely, given an essential simple arc $l$ in $S$ which is non-separating in $S$, one obtains the BP in $S$ whose curves are boundary components of a regular neighborhood of $l\cup \partial S$ in $S$. 

In Figure \ref{fig-hex3}, in place of BPs, essential simple arcs corresponding to them are drawn.
This replacement makes the drawing much plainer.
We define $a$, $c$ and $e$ as the h-curves in $S$ in Figure \ref{fig-hex3}, and define $b$, $d$ and $f$ as the BPs in $S$ corresponding to the arcs in Figure \ref{fig-hex3}.
Let $\Theta$ denote the hexagon in $\calt(S)$ of type 3 defined by the 6-tuple $(a, b, c, d, e, f)$.
Let $\alpha$, $\beta$ and $\gamma$ be the non-separating curves in Figure \ref{fig-ns}.
%====================================
\begin{figure}
\begin{center}
\includegraphics[width=10cm]{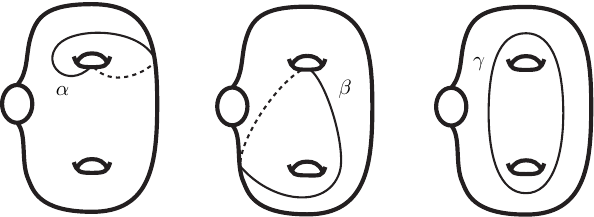}
\caption{}\label{fig-ns}
\end{center}
\end{figure}
%====================================
The curve $\alpha$ is disjoint from $a$ and $d$, the curve $\beta$ is disjoint from $b$ and $e$, and the curve $\gamma$ is disjoint from $c$ and $f$.
The following lemma is in contrast with Theorem \ref{thm-hex1-hex2} on hexagons of type 1 and type 2.

\begin{prop}\label{prop-hex3}
For any 3-path $L$ in $\Theta$, there exist infinitely many hexagons in $\calt(S)$ containing $L$.
\end{prop}

Before proving this proposition, we show the following:

\begin{lem}\label{lem-hex3}
Let $v_1$, $v_3$ and $v_5$ be h-vertices of $\calt(S)$, and $v_2$, $v_4$ and $v_6$ BP-vertices of $\calt(S)$ such that
\begin{itemize}
\item for any $j$ mod 6, $v_j$ and $v_{j+1}$ are adjacent, for any $k=1, 2$, $v_k$ and $v_{k+2}$ are distinct and not adjacent, and $v_1$ and $v_4$ are not adjacent; and
\item we have $v_6\neq v_2$.
\end{itemize}
Then the 6-tuple $(v_1,\ldots, v_6)$ defines a hexagon in $\calt(S)$.
\end{lem}

\begin{proof}
The assumption $v_6\neq v_2$ implies that $v_6$ and $v_2$ intersect.
Since $v_1$ is disjoint from $v_6$, but intersects $v_4$, the BPs $v_6$ and $v_4$ are distinct, and thus intersect.
In general, for any two distinct h-curves $x$, $y$ in $S$, there exists at most one BP in $S$ disjoint from $x$ and $y$ if it exists.
If $v_6$ and $v_3$ were disjoint, then the two BPs $v_6$ and $v_2$ would be disjoint from the two distinct h-curves $v_1$ and $v_3$.
This contradicts $v_6\neq v_2$.
It follows that $v_6$ and $v_3$ intersect.

Since $v_6$ is disjoint from $v_5$, but intersects $v_3$, the h-curves $v_5$ and $v_3$ are distinct.
The curves $v_5$ and $v_3$ intersect because they are disjoint from the BP $v_4$.
Since $v_4$ is disjoint from $v_5$, but intersects $v_1$, the h-curves $v_5$ and $v_1$ are distinct.
The curves $v_5$ and $v_1$ intersect because they are disjoint from the BP $v_6$.
If $v_5$ and $v_2$ were disjoint, then the two BPs $v_6$ and $v_2$ would be disjoint from the two distinct h-curves $v_1$ and $v_5$.
As noted in the previous paragraph, this contradicts $v_6\neq v_2$.
It follows that $v_5$ and $v_2$ intersect.
\end{proof}

\begin{proof}[Proof of Proposition \ref{prop-hex3}]
We prove the proposition in the case where $L$ consists of $a$, $b$, $c$ and $d$.
The proof of the other cases are obtained along a verbatim argument after exchanging symbols appropriately.
We show that for all but one of integers $n$, the 6-tuple $(a, b, c, d, t^n_{\alpha}(e), t_{\alpha}^n(f))$ defines a hexagon in $\calt(S)$.
For any integer $n$, the pair $\{ t^n_{\alpha}(e), t_{\alpha}^n(f)\}$ is an edge of $\calt(S)$.
Since $\alpha$ is disjoint from $a$ and $d$, we have $t_{\alpha}^n(a)=a$ and $t_{\alpha}^n(d)=d$.
Each of $\{ d, t^n_{\alpha}(e)\}$ and $\{ t^n_{\alpha}(f), a\}$ is thus an edge of $\calt(S)$.
To prove the proposition, it suffices to show the following three assertions:
\begin{enumerate}
\item[(1)] At most one integer $n$ satisfies the equality $t_{\alpha}^n(f)=b$.
\item[(2)] For any integer $n$ with $t_{\alpha}^n(f)\neq b$, the 6-tuple $(a, b, c, d, t^n_{\alpha}(e), t_{\alpha}^n(f))$ defines a hexagon in $\calt(S)$.
\item[(3)] For any integers $n_1$, $n_2$ with $n_1\neq n_2$, we have $t_{\alpha}^{n_1}(f)\neq t_{\alpha}^{n_2}(f)$.
\end{enumerate}
Assertions (1) and (3) hold because $\alpha$ and $f$ intersect.
Applying Lemma \ref{lem-hex3} when $(v_1,\ldots, v_6)=(a, b, c, d, t_{\alpha}^n(e), t_{\alpha}^n(f))$, we obtain assertion (2).
\end{proof}

The following lemma will be used in Section \ref{sec-aut}.

\begin{lem}\label{seq_hexagon}
Let $\Theta$ be the hexagon in $\calt(S)$ drawn in Figure \ref{fig-hex3}. 
For any two vertices $u$, $v$ of $\calt(S)$, there exists a sequence of hexagons in $\calt(S)$, $\Pi_1, \Pi_2, \ldots, \Pi_n$, satisfying the following three conditions: 
\begin{itemize}
\item For any $k=1,\ldots,n$, there exists $\gamma \in \mod(S)$ with $\Pi_k=\gamma(\Theta)$.
\item We have $u \in \Pi_1$ and $v \in \Pi_n$.
\item For any $k=1,\ldots,n-1$, the intersection $\Pi_k\cap \Pi_{k+1}$ contains a 2-path. 
\end{itemize}
\end{lem}

\begin{proof}
Pick two vertices $u$, $v$ of $\calt(S)$.
For any $\gamma \in \mod(S)$, the lemma holds for $u$ and $v$ if and only if it holds for $\gamma u$ and $\gamma v$.
To prove the lemma, we may therefore assume that $u$ is a vertex of $\Theta$.
For $j=1,\ldots, 5$, let $t_j$ denote the Dehn twist about the curve $\alpha_j$ drawn in Figure \ref{fig-gen}.
%====================================
\begin{figure}
\begin{center}
\includegraphics[width=3cm]{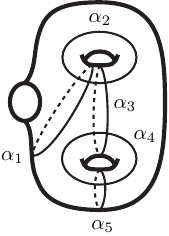}
\caption{}\label{fig-gen}
\end{center}
\end{figure}
%====================================
We set $T=\{t_1^{\pm 1}, \ldots, t_5^{\pm 1}\}$.
The group $\mod(S)$ is known to be generated by elements of $T$ (see \cite{gervais}). 
Since $\mod(S)$ transitively acts on $V_s(S)$ and on $V_{bp}(S)$, we can find an element $h$ of $\mod(S)$ with $v\in \{ h(a), h(b)\}$, where $a$ and $b$ are the h-vertex and the BP-vertex of $\Theta$, respectively, drawn in Figure \ref{fig-hex3}.
Write $h$ as a product $h=h_1\cdots h_n$ so that $h_j\in T$ for any $j$.  
For any $r\in T$, the intersection $r(\Theta)\cap \Theta$ contains a 2-path.
The sequence of hexagons in $\calt(S)$,
\[\Theta,\ h_1(\Theta),\ h_1h_2(\Theta),\ \ldots,\ h_1\cdots h_n(\Theta)=h(\Theta),\]
is thus a desired one.  
\end{proof}

%%%%%%%%%%%%%%%%%%%%%%%%%%%%%%%%%%%%%%%%%%%%

\section{Construction of an automorphism of the complex of curves}\label{sec-aut}

Throughout this section, we set $S=S_{2,1}$.
For any superinjective map $\phi$ from $\calt(S)$ into itself, we construct an automorphism of $\calc(S)$ inducing $\phi$.

\subsection{Surjectivity of a superinjective map}

In this subsection, we show that any superinjective map from $\calt(S)$ into itself preserves h-vertices and BP-vertices, respectively, and is surjective.

\begin{lem}\label{invariant_2_kind}
Let $\Theta$ be the hexagon in $\calt(S)$ drawn in Figure \ref{fig-hex3}. 
Then for any superinjective map $\phi \colon \calt(S) \rightarrow \calt(S)$ and any $\gamma \in \mod(S)$, the hexagon $\phi(\gamma(\Theta))$ in $\calt(S)$ is of type 3.  
\end{lem}

\begin{proof}
Pick $\gamma \in \mod(S)$.
The same property as that in Proposition \ref{prop-hex3} is satisfied by the hexagon $\gamma(\Theta)$, and hence by the image $\phi(\gamma(\Theta))$ because $\phi$ is superinjective.
By Theorem \ref{thm-hex1-hex2}, the hexagon $\phi(\gamma(\Theta))$ is of neither type 1 nor type 2, and is thus of type 3.
\end{proof}

\begin{lem}\label{lem-phi-pre}
Any superinjective map from $\calt(S)$ into itself preserves h-vertices and BP-vertices of $\calt(S)$, respectively.   
\end{lem}

\begin{proof}
Let $\phi \colon \calt(S)\rightarrow \calt(S)$ be a superinjective map.
Assuming that there exists an h-vertex $u$ of $\calt(S)$ with $\phi(u)$ a BP-vertex, we deduce a contradiction.
Pick an h-vertex $v$ of $\calt(S)$.
By Lemma \ref{seq_hexagon}, there exists a sequence of hexagons in $\calt(S)$, $\Pi_1, \Pi_2,\ldots , \Pi_n$, such that
\begin{itemize}
\item any $\Pi_k$ is of the form $\gamma(\Theta)$ for some $\gamma \in \mod(S)$;
\item we have $u \in \Pi_1$ and $v \in \Pi_n$; and
\item the intersection $\Pi_k\cap \Pi_{k+1}$ contains a 2-path for any $k=1,\ldots, n-1$.
\end{itemize} 
We note that for any $k=1,\ldots, n$, the hexagon $\phi(\Pi_k)$ is of type 3 by Lemma \ref{invariant_2_kind}, and that any edge of a hexagon in $\calt(S)$ of type 3 consists of an h-vertex and a BP-vertex.
Since $u$ is an h-vertex and $\phi(u)$ is a BP-vertex, the map $\phi$ sends h-vertices of $\Pi_1$ to BP-vertices of $\phi(\Pi_1)$, and sends BP-vertices of $\Pi_1$ to h-vertices of $\phi(\Pi_1)$.
Using the property that $\Pi_k\cap \Pi_{k+1}$ contains a 2-path for any $k=1,\ldots, n-1$, we inductively see that for any $k=1,\ldots, n$, the map $\phi$ sends h-vertices of $\Pi_k$ to BP-vertices of $\phi(\Pi_k)$, and sends BP-vertices of $\Pi_k$ to h-vertices of $\phi(\Pi_k)$. 
It turns out that $\phi(v)$ is a BP-vertex. 

We have shown that $\phi$ sends any h-vertex to a BP-vertex. 
It therefore follows that $\phi$ sends an edge consisting of two h-vertices to an edge consisting of two BP-vertices. 
This is a contradiction because $\calt(S)$ contains no edge consisting of two BP-vertices. 

We can also deduce a contradiction along a verbatim argument if we assume that there exists a BP-vertex $u$ of $\calt(S)$ with $\phi(u)$ an h-vertex. 
\end{proof}

We set $Y=S_{1,2}$. 
To prove surjectivity of a superinjective map from $\calt(S)$ into itself, we recall the following simplicial complexes associated to $Y$. 

\medskip

\noindent {\bf Complex $\cala(Y)$.} We define $\cala(Y)$ to be the abstract simplicial complex such that the set of vertices of $\cala(Y)$ is $V_a(Y)$, and a non-empty finite subset $\sigma$ of $V_a(Y)$ is a simplex of $\cala(Y)$ if and only if there exist mutually disjoint representatives of elements of $\sigma$. 

\medskip

\noindent {\bf Complex $\cald(Y)$.} We define $\cald(Y)$ to be the full subcomplex of $\cala(Y)$ spanned by all vertices that correspond to essential simple arcs in $Y$ connecting the two boundary components  of $Y$.

\begin{rem}\label{rem-d}
Let us describe simplices of $\cald(Y)$ of maximal dimension.
We denote by $Y_0$ the surface obtained from $Y$ by shrinking each component of $\partial Y$ to a point.
Let $P=\{ x_1, x_2\}$ denote the set of the two points of $Y_0$ into which components of $\partial Y$ are shrunken.
The natural map from $Y$ onto $Y_0$ induces the bijection from $V_a(Y)$ onto the set of isotopy classes of ideal arcs in the punctured surface $(Y_0, P)$.
It turns out that a simplex of $\cala(Y)$ of maximal dimension corresponds to an ideal triangulation of $(Y_0, P)$, and that a simplex of $\cald(Y)$ of maximal dimension corresponds to an ideal squaring of $(Y_0, P)$ defined as follows.
We mean by an {\it ideal squaring} of $(Y_0, P)$ a cell division $\delta$ of $Y_0$ such that
\begin{itemize}
\item the set of 0-cells of $\delta$ is $P$;
\item any 1-cell of $\delta$ is an ideal arc in $(Y_0, P)$ connecting $x_1$ and $x_2$; and
\item any 2-cell of $\delta$ is a {\it square}, that is, it is obtained by attaching a Euclidean square $\tau$ to the 1-skeleton of $\delta$, mapping each vertex of $\tau$ to a 0-cell of $\delta$, and each edge of $\tau$ to a 1-cell of $\delta$.
\end{itemize}
By argument on the Euler characteristic of $Y_0$, for any ideal squaring $\delta$ of $(Y_0, P)$, the numbers of 1-cells and 2-cells of $\delta$ are equal to 4 and 2, respectively.
\end{rem}

We will use the following:

\begin{prop}\label{prop-d-surj}
We set $Y=S_{1,2}$.
Then any injective simplicial map from $\cald(Y)$ into itself is surjective.
\end{prop}

This proposition follows from \cite[Proposition 3.1 and Lemma 3.2]{kida-cohop}.
For a vertex $v$ of $\calt(S)$, we denote by $\lk(v)$ the link of $v$ in $\calt(S)$.

\begin{lem}\label{surj_lkb}
Let $b$ be a BP-vertex of $\calt(S)$ and $\phi \colon \calt(S)\rightarrow \calt(S)$ a superinjective map. 
Then the equality $\phi(\lk(b))=\lk(\phi(b))$ holds. 
\end{lem}

\begin{proof}
We may assume $\phi(b)=b$.
Let $Y$ denote the component of $S_b$ homeomorphic to $S_{1, 2}$.
As noted right after Lemma \ref{lem-13}, there is a one-to-one correspondence between elements of $V_s(Y)$ and elements of $V_a(Y)$ whose representatives connect the two components of $\partial Y$.
For $\alpha \in V_s(Y)$, we denote by $l_{\alpha}$ the element of $V_a(Y)$ corresponding to $\alpha$.
By Theorem \ref{thm-type1} (ii), for any two distinct vertices $\alpha, \beta \in V_s(Y)$, there is a hexagon in $\calt(S)$ of type 1 containing $b$, $\alpha$ and $\beta$ if and only if $l_{\alpha}$ and $l_{\beta}$ are disjoint (see Figure \ref{fig-link} (a) for such two disjoint arcs).
%====================================
\begin{figure}
\begin{center}
\includegraphics[width=11cm]{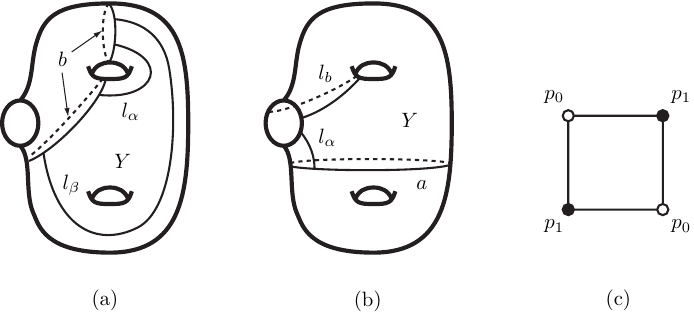}
\caption{}\label{fig-link}
\end{center}
\end{figure}
%====================================
The map $\phi$ induces an injective simplicial map from $\cald(Y)$ into itself because $\phi$ preserves hexagons in $\calt(S)$ of type 1 by Lemma \ref{lem-phi-pre}.
Proposition \ref{prop-d-surj} implies the equality in the lemma.
\end{proof}

\begin{lem}\label{surj_lkd}
Let $a$ be an h-vertex of $\calt(S)$ and $\phi \colon \calt(S)\rightarrow \calt(S)$ a superinjective map. 
Then the equality $\phi(\lk(a))=\lk(\phi(a))$ holds. 
\end{lem}

\begin{proof}
We may assume $\phi(a)=a$.
Let $Y$ denote the component of $S_a$ homeomorphic to $S_{1, 2}$.
Note that $V_s(Y)$ and $V_{bp}(Y)$ are naturally identified with sets of vertices of $\lk(a)$.
An argument similar to the proof of Lemma \ref{surj_lkb} shows that $\phi$ induces an injective simplicial map $\tilde{\phi}\colon \cald(Y)\rightarrow \cald(Y)$, which is surjective by Proposition \ref{prop-d-surj}.
It follows that $\phi$ sends $V_s(Y)$ onto itself.

We prove that $\phi$ sends $V_{bp}(Y)$ onto itself.
As noted in the second paragraph of Section \ref{sec-type3}, there is a one-to-one correspondence between elements of $V_{bp}(Y)$ and elements of $V_a(Y)$ whose representatives are non-separating in $Y$ and connect two points in the component of $\partial Y$ corresponding to $\partial S$.
For $b \in V_{bp}(Y)$, we denote by $l_b$ the element of $V_a(Y)$ corresponding to $b$.
We use the same symbol as in the proof of Lemma \ref{surj_lkb}.
Namely, for $\alpha \in V_s(Y)$, we denote by $l_{\alpha}$ the element of $V_a(Y)$ corresponding to $\alpha$.
By Theorem \ref{thm-type1} (ii), for any $b\in V_{bp}(Y)$ and $\alpha \in V_s(Y)$, there exists a hexagon in $\calt(S)$ of type 1 containing $a$, $b$ and $\alpha$ if and only if $l_b$ and $l_{\alpha}$ are disjoint (see Figure \ref{fig-link} (b) for such two disjoint arcs).

Pick a simplex $\sigma$ of $\cald(Y)$ of maximal dimension.
Let $Y_0$ denote the surface obtained from $Y$ by shrinking each component of $\partial Y$ to a point.
Let $P$ denote the set of points of $Y_0$ into which components of $\partial Y$ are shrunken.
We then obtain the punctured surface $(Y_0, P)$.
Let $p_0$ denote the point of $P$ into which $\partial S$ is shrunken.
Let $p_1$ denote the other point of $P$.
As discussed in Remark \ref{rem-d}, we have the ideal squaring of $(Y_0, P)$ corresponding to $\sigma$, and the set of whose 2-cells consists of two squares.
In each of those two squares, as in Figure \ref{fig-link} (c), two opposite vertices correspond to $p_0$, the other two vertices correspond to $p_1$, and we have an arc connecting the two vertices corresponding to $p_0$ and dividing the square into two triangles.
It follows from Lemma \ref{lem-mosher} that up to isotopy, there exist exactly two ideal arcs in $(Y_0, P)$ disjoint from any ideal arc corresponding to an element of $\sigma$ and both of whose end points are $p_0$. 
We define $L(\sigma)$ as the subset of $V_a(Y)$ consisting of the two elements that correspond to those ideal arcs.
Any arc in $L(\sigma)$ is non-separating in $Y$ because for any essential simple arc $l$ in $Y$ that is separating in $Y$, a vertex of $\cald(Y)$ whose representative is disjoint from $l$ uniquely exists and because we have $|\sigma|=4$.

The claim in the end of the second paragraph of the proof and injectivity of $\phi$ imply that for any simplex $\sigma$ of $\cald(Y)$ of maximal dimension, the map $\phi$ induces a bijection from $L(\sigma)$ onto $L(\tilde{\phi}(\sigma))$.
For any $b\in V_{bp}(Y)$, there exists a simplex of $\cald(Y)$ of maximal dimension any of whose arcs is disjoint from $l_b$.  
Surjectivity of the map $\tilde{\phi}\colon \cald(Y)\rightarrow \cald(Y)$ therefore implies that $\phi$ sends $V_{bp}(Y)$ onto itself.
\end{proof}

The last two lemmas and connectivity of $\calt(S)$ imply the following: 

\begin{thm}\label{surjective}
Any superinjective map from $\calt(S)$ into itself is surjective and is thus an automorphism of $\calt(S)$. 
\end{thm}

%%%%%%%%%%%%%%%%%%%%%%%%%%%%%%%%%%%%%%%

\subsection{Construction of a map from $V(S)$ into itself}

Let $\phi$ be an automorphism of $\calt(S)$.
We define a map $\Phi \colon V(S) \rightarrow V(S)$ as follows. 
Pick an element $\alpha$ of $V(S)$.
If $\alpha$ is separating in $S$, then we set $\Phi(\alpha)=\phi(\alpha)$. 
If $\alpha$ is non-separating in $S$, then pick a hexagon $\Pi$ in $\calt(S)$ of type 1 such that $\alpha$ is contained in the two BP-vertices of $\Pi$, and define $\Phi(\alpha)$ to be the non-separating curve in $S$ contained in the two BP-vertices of the hexagon $\phi(\Pi)$ of type 1. 

We will prove that $\Phi$ is well-defined as a consequence of Lemma \ref{well-def}.
To prove it, let us introduce the following:

\medskip

\noindent {\bf Graph $\cale$.} We define $\cale$ to be the simplicial graph so that the set of vertices of $\cale$ is $V_{bp}(S)$, and two distinct vertices $u$, $v$ of $\cale$ are connected by an edge of $\cale$ if and only if there exists a hexagon in $\calt(S)$ of type 1 containing $u$ and $v$.

\medskip

We mean by a {\it square} in $\cale$ the full subgraph of $\cale$ spanned by exactly four vertices $v_1,\ldots , v_4$ such that for any $k$ mod $4$, $v_k$ and $v_{k+1}$ are adjacent, and $v_k$ and $v_{k+2}$ are not adjacent. 
In this case, let us say that the square is defined by the 4-tuple $(v_{1},\ldots, v_{4})$.

\begin{lem}\label{square}
Let $(v_1,\ldots,v_4)$ be a 4-tuple defining a square in $\cale$. 
Then there exists a non-separating curve $\alpha$ in $S$ with $\alpha \in v_k$ for any $k=1,\ldots,4$. 
\end{lem}

\begin{proof}
By the definition of $\cale$, for any two adjacent vertices of $\cale$, the two BPs in $S$ corresponding to them share a non-separating curve in $S$.
For each $k$ mod $4$, let $\beta_k$ denote the non-separating curve in $S$ contained in $v_k$ and $v_{k+1}$.
Without loss of generality, it suffices to deduce a contradiction under the assumption $\beta_1\neq \beta_2$.

Let $\bar{S}$ denote the closed surface obtained from $S$ by attaching a disk to $\partial S$, and let $\pi \colon \calc(S)\rightarrow \calc(\bar{S})$ be the simplicial map associated with the inclusion of $S$ into $\bar{S}$.
Since $\pi$ sends the two curves in any BP in $S$ to the same curve in $\bar{S}$, all curves in the BPs $v_1,\ldots, v_4$ are sent to the same curve in $\bar{S}$, denoted by $\alpha_0$.
In other words, all curves in the BPs $v_1,\ldots, v_4$ are in $\pi^{-1}(\alpha_0)$.
Let $T$ denote the full subcomplex of $\calc(S)$ spanned by $\pi^{-1}(\alpha_0)$, which is a tree by Theorem \ref{thm-tree}.
The sequence, $\beta_1$, $\beta_2$, $\beta_3$, $\beta_4$, $\beta_1$, forms a closed path in $T$.

We assume $\beta_1\neq \beta_2$.
The equality $v_2=\{ \beta_1, \beta_2\}$ then holds.
We have $\beta_3\neq \beta_1$ and $\beta_4\neq \beta_2$ because $v_3\neq v_2$ and $v_1\neq v_2$.
Let $\gamma$ and $\delta$ denote the curves in $S$ with $v_1=\{ \beta_1, \gamma \}$ and $v_3=\{ \beta_2, \delta \}$.
Each of $\gamma$ and $\delta$ is equal to neither $\beta_1$ nor $\beta_2$ because $v_1\neq v_2$ and $v_3\neq v_2$.
We have either $\beta_4=\gamma$ or $\beta_3=\delta$ because otherwise we would have $v_2=v_4$.

If $\beta_4=\gamma$, then we have $\beta_3\neq \beta_2$ and $\beta_3\neq \gamma$ because otherwise the sequence, $\beta_1$, $\beta_2$, $\gamma$, $\beta_1$, would form a simple closed path in $T$.
It turns out that $\beta_1$, $\beta_2$, $\beta_3$ and $\beta_4$ are mutually distinct.
This is a contradiction.

If $\beta_3=\delta$, then we have $\beta_4\neq \beta_1$ and $\beta_4\neq \delta$ because otherwise the sequence, $\beta_1$, $\beta_2$, $\delta$, $\beta_1$, would form a simple closed path in $T$.
It turns out that $\beta_1$, $\beta_2$, $\beta_3$ and $\beta_4$ are mutually distinct.
This is also a contradiction.
\end{proof}

\begin{lem}\label{well-def}
Let $\phi$ be an automorphism of $\calt(S)$.
Let $\alpha$ be a non-separating curve in $S$.
Pick two hexagons $\Pi$, $\Omega$ in $\calt(S)$ of type 1 such that any BP-vertex of $\Pi$ and $\Omega$ contains $\alpha$.
Then the non-separating curve in $S$ contained in the two BP-vertices of $\phi(\Pi)$ is equal to that of $\phi(\Omega)$.
\end{lem}

\begin{proof}
Let $a_1$ and $a_2$ denote the two BP-vertices of $\Pi$.
Let $b_1$ and $b_2$ denote the two BP-vertices of $\Omega$.

\begin{claim}\label{seq_square}
There exists a sequence of squares in $\cale$, $\Delta_1,\ldots,\Delta_n$, satisfying the following three conditions:
\begin{itemize}
\item The square $\Delta_1$ contains $a_1$ and $a_2$, and the square $\Delta_n$ contains $b_1$ and $b_2$. 
\item For any $k=1,\ldots,n$, any vertex of $\Delta_k$ contains $\alpha$.
\item For any $k=1,\ldots,n-1$, the intersection $\Delta_k\cap \Delta_{k+1}$ contains an edge of $\cale$.
\end{itemize}
\end{claim}

\begin{proof}
Let $\alpha_1$, $\alpha_2$, $\beta_1$ and $\beta_2$ be the curves in $S$ with $a_j=\{ \alpha, \alpha_j\}$ and $b_j=\{ \alpha, \beta_j\}$ for $j=1, 2$.
The curves $\alpha_1$ and $\alpha_2$ can be drawn as in Figure \ref{fig-gene} (a), where the surface $S_{\alpha}$ obtained by cutting $S$ along $\alpha$ is drawn.
We define $\gamma_1,\ldots, \gamma_4$ as the curves in $S$ drawn in Figure \ref{fig-gene} (b).
%====================================
\begin{figure}
\begin{center}
\includegraphics[width=8cm]{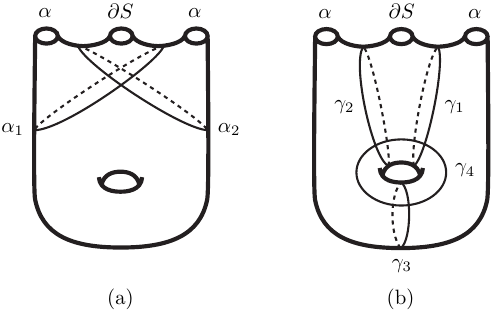}
\caption{}\label{fig-gene}
\end{center}
\end{figure}
%====================================
The equality $i(\alpha_1, \gamma_1)=i(\alpha_2, \gamma_2)=0$ then holds.
For $j=1,\ldots, 4$, let $t_j\in \mod(S)$ denote the Dehn twist about $\gamma_j$.
Let $\mod(S)_{\alpha}$ denote the stabilizer of $\alpha$ in $\mod(S)$, and define
\[q\colon \mod(S)_{\alpha}\rightarrow \mod(S_{\alpha})\]
as the natural homomorphism.
The group $\pmod(S_\alpha)$ is known to be generated by $q(t_1),\ldots, q(t_4)$ (see \cite{gervais}).
By Theorem \ref{thm-13}, there exists an element $h$ of $\pmod(S_{\alpha})$ with $\{ h(\alpha_1), h(\alpha_2)\}=\{ \beta_1, \beta_2\}$.
Write $h$ as a product $h=q(h_1)\cdots q(h_n)$ so that $h_j\in \{ t_1^{\pm 1},\ldots, t_4^{\pm 1}\}$ for any $j$.

The 4-tuple $(a_1, a_2, t_2(a_1), t_1(a_2))$ defines a square in $\cale$. We denote it by $\Delta$.
For any $w\in \{ t_1^{\pm 1},\ldots, t_4^{\pm 1}\}$, the intersection $\Delta \cap w(\Delta)$ contains an edge of $\cale$.
The sequence of squares in $\cale$,
\[\Delta,\ h_1(\Delta),\ h_1h_2(\Delta),\ \ldots ,\ h_1h_2\cdots h_n(\Delta)=h(\Delta),\]
is therefore a desired one.
\end{proof}

By the definition of $\cale$, the automorphism $\phi$ of $\calt(S)$ induces an automorphism of $\cale$.
If $\Delta_1,\ldots, \Delta_n$ are the squares in $\cale$ chosen in Claim \ref{seq_square}, then $\phi(\Delta_1),\ldots, \phi(\Delta_n)$ are also squares in $\cale$ such that
\begin{itemize}
\item $\phi(\Delta_1)$ contains $\phi(a_1)$ and $\phi(a_2)$, and $\phi(\Delta_n)$ contains $\phi(b_1)$ and $\phi(b_2)$; and
\item for any $k=1,\ldots, n-1$, the intersection $\phi(\Delta_k)\cap \phi(\Delta_{k+1})$ contains an edge of $\cale$.
\end{itemize}
It follows from Lemma \ref{square} that for any $k=1,\ldots, n$, there exists a non-separating curve in $S$ contained in any vertex of $\phi(\Delta_k)$.
The above second condition implies that this curve does not depend on $k$.
In particular, the curve shared by $\phi(a_1)$ and $\phi(a_2)$ is equal to the curve shared by $\phi(b_1)$ and $\phi(b_2)$.
\end{proof}

Lemma \ref{well-def} implies that the map $\Phi \colon V(S)\rightarrow V(S)$ constructed in the beginning of this subsection is well-defined.

\begin{lem}\label{simplicial}
Let $\phi$ be an automorphism of $\calt(S)$, and let $\Phi \colon V(S) \rightarrow V(S)$ be the map defined in the beginning of this subsection.
Then $\Phi$ defines a simplicial map from $\calc(S)$ into itself. 
\end{lem}

\begin{proof}
Let $\alpha$ and $\beta$ be distinct elements of $V(S)$ with $i(\alpha,\beta)=0$. 
We have to show $i(\Phi(\alpha),\Phi(\beta))=0$. 
If both $\alpha$ and $\beta$ are separating in $S$, then this equality follows from the definition of $\Phi$ and simpliciality of $\phi$. 
If $\alpha$ and $\beta$ are non-separating curves in $S$ with $\{\alpha,\beta\}$ a BP in $S$, then $\Phi(\alpha)$ and $\Phi(\beta)$ are curves in the BP $\phi(\{\alpha, \beta\})$ by the definition of $\Phi$. 
The equality $i(\Phi(\alpha),\Phi(\beta))=0$ thus holds. 

For $\gamma \in V_s(S)$, we denote by $H_\gamma$ the component of $S_\gamma$ that is a handle.  
If $\alpha$ and $\beta$ are non-separating curves in $S$ such that $\{\alpha,\beta\}$ is not a BP in $S$, then there exist $\gamma, \delta \in V_s(S)$ with $\gamma \neq \delta$, $i(\gamma,\delta)=0$, $\alpha \in V(H_\gamma)$ and $\beta \in V(H_\delta)$.
We can then find two hexagons $\Pi$, $\Omega$ in $\calt(S)$ of type 1 such that
\begin{itemize}
\item the two BP-vertices of $\Pi$ contain $\alpha$, and those of $\Omega$ contain $\beta$; and 
\item both $\Pi$ and $\Omega$ contain the h-vertices $\gamma$ and $\delta$.
\end{itemize}
By the definition of $\Phi$, we have $\Phi(\alpha)\in V(H_{\phi(\gamma)})$ and $\Phi(\beta)\in V(H_{\phi(\delta)})$.
Since we have $\phi(\gamma)\neq \phi(\delta)$ and $i(\phi(\gamma),\phi(\delta))=0$, the equality $i(\Phi(\alpha),\Phi(\beta))=0$ holds. 

If $\alpha$ is non-separating in $S$ and $\beta$ is separating in $S$, then one can find two distinct BPs $a_1$, $a_2$ in $S$ containing $\alpha$ and a hexagon $\Pi$ in $\calt(S)$ of type 1 such that $\Pi$ contains $a_1$, $a_2$ and $\beta$ as its vertices. 
By the definition of $\Phi$, the curve $\Phi(\alpha)$ is disjoint from any curve corresponding to an h-vertex of $\phi(\Pi)$.
We therefore have $i(\Phi(\alpha),\Phi(\beta))=0$. 
\end{proof}

Let $\phi$ be an automorphism of $\calt(S)$.
We have constructed the simplicial map $\Phi \colon \calc(S)\rightarrow \calc(S)$ associated to $\phi$.
The simplicial map from $\calc(S)$ into itself associated to $\phi^{-1}$ is then the inverse of $\Phi$.
The map $\Phi$ is therefore an automorphism of $\calc(S)$ and is induced by an element of $\mod^*(S)$ by Theorem \ref{thm-iva}.
If $\{ \alpha, \beta \}$ is a BP in $S$, then $\Phi(\alpha)$ and $\Phi(\beta)$ are curves in the BP $\phi(\{ \alpha, \beta\})$ by the definition of $\Phi$, and are distinct because $\Phi$ is an automorphism of $\calc(S)$.
We thus have the equality $\{ \Phi(\alpha), \Phi(\beta)\}=\phi(\{ \alpha, \beta \})$.
It follows that $\phi$ is induced by an element of $\mod^*(S)$.
Combining this with Theorem \ref{surjective}, we obtain the following:

\begin{thm}
Any superinjective map from $\calt(S)$ into itself is induced by an element of $\mod^*(S)$. 
\end{thm}

%%%%%%%%%%%%%%%%%%%%%%%%%%%%%%%%%%%%%%%%

\appendix

\section{The minimal length of simple cycles}\label{sec-app}

Let $\cal{G}$ be a simplicial graph.
We mean by a {\it simple cycle} in $\cal{G}$ a subgraph of $\cal{G}$ obtained as the image of a simple closed path in $\cal{G}$.
A simple cycle in $\cal{G}$ is called {\it non-trivial} if its length is positive.

Throughout this appendix, we set $S=S_{2, 1}$.
We aim to show the following:

\begin{prop}\label{prop-five}
There exists no non-trivial simple cycle in $\calt(S)$ of length at most 5.
\end{prop}

It turns out that hexagons in $\calt(S)$ are simple cycles in $\calt(S)$ of minimal positive length.
We first prove the following:

\begin{lem}\label{lem-four}
There exists no non-trivial simple cycle in $\calt(S)$ of length at most 4.
\end{lem}

\begin{proof}
Since $\calt(S)$ is one-dimensional, there exists no simple cycle in $\calt(S)$ of length 3.
Assume that there are four vertices $v_1,\ldots, v_4$ of $\calt(S)$ with $i(v_j, v_{j+1})=0$ and $i(v_j, v_{j+2})\neq 0$ for any $j$ mod $4$.
We can find $\alpha_1,\ldots, \alpha_4\in V(S)$ such that
\begin{itemize}
\item for any $j=1,\ldots, 4$, we have either $\alpha_j=v_j\in V_s(S)$ or $v_j\in V_{bp}(S)$ and $\alpha_j\in v_j$; and
\item for any $k=1, 2$, we have $i(\alpha_k, \alpha_{k+2})\neq 0$.
\end{itemize}
For a surface $X$, we denote by $\chi(X)$ the Euler characteristic of $X$.
For $k=1, 2$, we define $Q_k$ as the subsurface of $S$ filled by $\alpha_k$ and $\alpha_{k+2}$.
If $|\chi(Q_k)|\geq 2$, then set $R_k=Q_k$.
If $|\chi(Q_k)|=1$, then $Q_k$ is a handle, and $\alpha_k$ and $\alpha_{k+2}$ are non-separating in $S$.
It follows that $v_k$ and $v_{k+2}$ are BPs in $S$.
The curve in $v_k$ distinct from $\alpha_k$, denoted by $\beta_k$, intersects the h-curve in $S$ corresponding to the boundary of $Q_k$.
In the case of $|\chi(Q_k)|=1$, we define $R_k$ as the subsurface of $S$ filled by the three curves $\alpha_k$, $\beta_k$ and $\alpha_{k+2}$. 
In both cases, $R_1$ and $R_2$ can be realized so that they are disjoint, and we have $|\chi(R_1)|\geq 2$ and $|\chi(R_2)|\geq 2$.
This contradicts $|\chi(S)|=3$.
The lemma follows.
\end{proof}

The proof of Proposition \ref{prop-five} reduces to showing the following:

\begin{lem}\label{lem-pen}
There exists no pentagon in $\calt(S)$.
\end{lem}

\begin{proof}
Let $\bar{S}$ denote the closed surface obtained from $S$ by attaching a disk to $\partial S$.
We have the simplicial maps
\[\pi \colon \calc(S)\rightarrow \calc(\bar{S}),\quad \theta \colon \calt(S)\rightarrow \calc(\bar{S})\]
associated with the inclusion of $S$ into $\bar{S}$.
Note that $\theta$ sends each BP-vertex of $\calt(S)$ to a vertex of $\calc(\bar{S})$ corresponding to a non-separating curve in $\bar{S}$, and that both $\pi$ and $\theta$ send any two adjacent h-vertices to the same h-vertex of $\calc(\bar{S})$.
Since the fiber of $\pi$ over each vertex of $\calc(\bar{S})$ is a tree by Theorem \ref{thm-tree}, there exists no pentagon in $\calt(S)$ consisting of only h-vertices.
We thus have to show non-existence of pentagons in $\calt(S)$ having one or two BP-vertices.

\begin{claim}\label{claim-pen-2bps}
There exists no pentagon in $\calt(S)$ defined by a 5-tuple $(a, b, c, d, e)$ such that
$a$, $c$ and $e$ are h-vertices and $b$ and $d$ are BP-vertices.
\end{claim}

\begin{proof}
Assuming that such a 5-tuple $(a, b, c, d, e)$ exists, we deduce a contradiction.
Choose representatives $A,\ldots, E$ of $a,\ldots, e$, respectively, such that any two of them intersect minimally.
Let $R$ denote the component of $S_C$ that is not a handle.
Since $B$ is a BP disjoint from $A$ and $C$, the intersection $A\cap R$ consists of mutually isotopic, essential simple arcs in $R$ which are non-separating in $R$.
Similarly, $E\cap R$ also consists of mutually isotopic, essential simple arcs in $R$ which are non-separating in $R$.
Let $l_A$ be a component of $A\cap R$, and let $l_E$ be a component of $E\cap R$.
The arcs $l_A$ and $l_E$ are not isotopic because otherwise the equality $b=d$ would hold.
Since $A$ and $E$ are disjoint, $l_A$ and $l_E$ are also disjoint.

We first assume that along $C$, the end points of $l_A$ first appear, and those of $l_E$ then appear.
Cut $R$ along $l_A$.
We then obtain a pair of pants.
Up to a homeomorphism of that pair of pants fixing points of its boundary, the arc $l_E$ is drawn as in Figure \ref{fig-pants} (a) or (b).
%====================================
\begin{figure}
\begin{center}
\includegraphics[width=12cm]{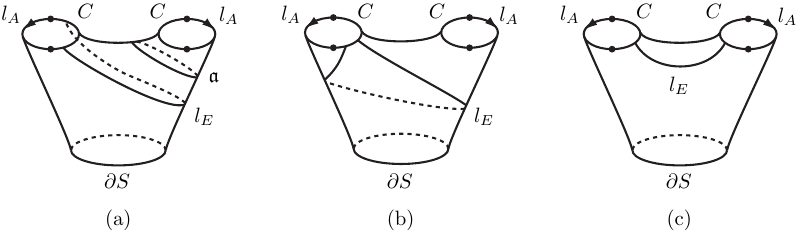}
\caption{The pair of pants obtained by cutting $R$ along $l_A$}\label{fig-pants}
\end{center}
\end{figure}
%====================================
In Figure \ref{fig-pants} (b), the union of $l_E$ and a subarc of $C$ cuts off an annulus containing $\partial S$ from $S$, and this contradicts that $l_E$ is non-separating in $R$.
The arc $l_E$ is thus drawn as in Figure \ref{fig-pants} (a).
The curve $\mathfrak{a}$ in Figure \ref{fig-pants} (a) is a boundary component of a regular neighborhood of $l_A\cup C$ in $S$, and is also that of $l_E\cup C$.
It follows that the isotopy class of $\mathfrak{a}$ is contained in $b$ and $d$.
The surface $S_{\mathfrak{a}}$ obtained by cutting $S$ along $\mathfrak{a}$ is homeomorphic to $S_{1, 3}$.
The curve $A$ is an h-curve in $S_{\mathfrak{a}}$ because $A$ is disjoint from the BP $B$ in $S$.
Similarly, $E$ is also an h-curve in $S_{\mathfrak{a}}$ because $E$ is disjoint from the BP $D$ in $S$.
Since $A$ and $E$ are disjoint h-curves in $S_{\mathfrak{a}}$, they have to be isotopic.
This is a contradiction.

We next assume that along $C$, the end points of $l_A$ and $l_E$ appear alternately.
Cut $R$ along $l_A$.
We then obtain a pair of pants.
Up to a homeomorphism of that pair of pants fixing points of its boundary, the arc $l_E$ is drawn as in Figure \ref{fig-pants} (c).
We then have $i(\theta(b), \theta(d))=1$.
The curves $\theta(b)$ and $\theta(d)$ fill a component of $\bar{S}_{\theta(c)}$.
The equality $\theta(a)=\theta(e)$ holds because $a$ and $e$ are adjacent h-vertices.
It follows that $\theta(a)$ is an h-curve in $\bar{S}$ disjoint from $\theta(b)$ and $\theta(d)$.
We thus have the equality $\theta(a)=\theta(c)$.
On the other hand, $A$ and $C$ are curves in the component of $S_B$ that does not contain $\partial S$.
The equality $\theta(a)=\theta(c)$ implies the equality $a=c$.
This is a contradiction.
\end{proof}

\begin{claim}\label{claim-pen-1bp}
There exists no pentagon in $\calt(S)$ defined by a 5-tuple $(a, b, c, d, e)$ such that
$a$, $c$, $d$ and $e$ are h-vertices and $b$ is a BP-vertex.
\end{claim}

\begin{proof}
Assume that such a 5-tuple $(a, b, c, d, e)$ exists.
We have the equality $\theta(a)=\theta(e)=\theta(d)=\theta(c)$.
We can deduce a contradiction along the argument in the end of the proof of Claim \ref{claim-pen-2bps}.
\end{proof}

Claims \ref{claim-pen-2bps} and \ref{claim-pen-1bp} complete the proof of Lemma \ref{lem-pen}. 
\end{proof}

%%%%%%%%%%%%%%%%%%%%%%%%%%%%%%%%%%%%%%%%%%%%%


\begin{thebibliography}{99}

\bibitem{bbm}M. Bestvina, K.-U. Bux and D. Margalit, The dimension of the Torelli group, {\it J. Amer. Math. Soc.} {\bf 23} (2010), 61--105.

\bibitem{birman}J. S. Birman, {\it Braids, links, and mapping class groups}, Ann. of Math. Stud., 82, Princeton Univ.\ Press, Princeton, N.J., 1974.

\bibitem{bm}T. Brendle and D. Margalit, Commensurations of the Johnson kernel, {\it Geom. Topol.} {\bf 8} (2004), 1361--1384.

\bibitem{bm-add}T. Brendle and D. Margalit, Addendum to: Commensurations of the Johnson kernel, {\it Geom. Topol.} {\bf 12} (2008), 97--101.

\bibitem{farb-ivanov}B. Farb and N. V. Ivanov, The Torelli geometry and its applications: research announcement, {\it Math. Res. Lett.} {\bf 12} (2005), 293--301.

\bibitem{flp}A. Fathi, F. Laudenbach and V. Po\'enaru, {\it Travaux de Thurston sur les surfaces}. S\'eminaire Orsay, Ast\'erisque, 66--67. Soc. Math. France, Paris, 1979.

\bibitem{gervais}S. Gervais, A finite presentation of the mapping class group of a punctured surface, {\it Topology} {\bf 40} (2001), 703--725. 

\bibitem{irmak1}E. Irmak, Superinjective simplicial maps of complexes of curves and injective homomorphisms of subgroups of mapping class groups, {\it Topology} {\bf 43} (2004), 513--541.

\bibitem{iva-aut}N. V. Ivanov, Automorphisms of complexes of curves and of Teichm\"uller spaces, {\it Int. Math. Res. Not.} {\bf 1997}, no. 14, 651--666. 

\bibitem{johnson}D. L. Johnson, Homeomorphisms of a surface which act trivially on homology, {\it Proc. Amer. Math. Soc.} {\bf 75} (1979), 119--125.

\bibitem{kls}R. P. Kent IV, C. J. Leininger and S. Schleimer, Trees and mapping class groups, {\it J. Reine Angew. Math.} {\bf 637} (2009), 1--21.

\bibitem{kida-tor}Y. Kida, Automorphisms of the Torelli complex and the complex of separating curves, {\it J. Math. Soc. Japan} {\bf 63} (2011), 363--417.

\bibitem{kida-cohop}Y. Kida, The co-Hopfian property of the Johnson kernel and the Torelli group, preprint, arXiv:0911.3923, to appear in {\it Osaka J. Math.}

\bibitem{kork-aut}M. Korkmaz, Automorphisms of complexes of curves on punctured spheres and on punctured tori, {\it Topology Appl.} {\bf 95} (1999), 85--111.

\bibitem{luo}F. Luo, Automorphisms of the complex of curves, {\it Topology} {\bf 39} (2000), 283--298.

\bibitem{mv}J. D. McCarthy and W. R. Vautaw, Automorphisms of Torelli groups, preprint, arXiv:math/0311250.

\bibitem{mess}G. Mess, The Torelli group for genus 2 and 3 surfaces, {\it Topology} {\bf 31} (1992), 775--790.

\bibitem{mosher}L. Mosher, Tiling the projective foliation space of a punctured surface, {\it Trans. Amer. Math. Soc.} {\bf 306} (1988), 1--70.

\bibitem{putman-conn}A. Putman, A note on the connectivity of certain complexes associated to surfaces, {\it Enseign. Math. (2)} {\bf 54} (2008), 287--301. 

\end{thebibliography}
\end{document}